\documentclass{article}[12pt]
\usepackage{xcolor}
\usepackage{graphicx} 
\usepackage[margin=3cm]{geometry}
\usepackage{amsmath, xcolor, amsfonts, amssymb, amsthm, mathrsfs}
\usepackage{amsmath, amsthm, amssymb, bm, graphicx,  mathrsfs,amsfonts,amssymb,amsxtra,mathabx,romannum,caption, indentfirst}\usepackage{pst-node,bbold}
\usepackage{graphicx} 
\usepackage{amsmath,dsfont}
\usepackage{amssymb}
\usepackage{xcolor} 
\usepackage{mathtools}
\usepackage{colortbl} 
\usepackage{tikz} 
\usepackage{hf-tikz} 
\usepackage{makecell}
\usepackage{mathrsfs}
\usepackage{graphicx}
\usepackage{enumerate}
\usepackage{chngcntr}
\usepackage{enumitem}
\usepackage[colorlinks,
linkcolor=red,
anchorcolor=blue,
citecolor=green
]{hyperref}

\usepackage[english]{babel}
\usepackage{mathtools}
 \usepackage{nccmath}
 \newtheorem{theorem}{Theorem}[section]
\newtheorem{definition}[theorem]{Definition}
\newtheorem{lemma}[theorem]{Lemma}

\newtheorem{prop}[theorem]{Proposition}

\usepackage[]{amssymb}
\newcommand{\R}{\mathbb{R}}

\newcommand{\E}{E}

\newcommand{\N}{\mathbb{N}}

\newcommand{\Z}{\mathbb{Z}}
\newcommand{\D}{\mathcal{D}}

\newcommand{\T}{\mathbb{T}}

\title{Fourier restriction to hyperbolic  rectangles and an application}
\author{Kaiwen Jin}
\date{\today}

\begin{document}
\pagenumbering{arabic}
\maketitle
\begin{abstract}
In this article, we study the Lebesgue space inequalities for extension operators associated with hyperbolic surfaces over rectangular regions. We characterize the corresponding operator norms in terms of the side-lengths. As an application, we present new restriction estimates for a class of hypersurfaces with additive structure.    
\end{abstract} 

\bigskip

\section{Introduction}
Given a surface \[S:=\{(\xi,\phi(\xi)),\xi\in B_1^n\},\] the Fourier extension operator is defined as\[Ef(x,t):=\int_{B_1^n}f(\xi)e^{2\pi i(x\cdot \xi+t\phi(\xi))}d\xi.\]Here $B^n_1=\{\xi\in\R^n:|\xi|\le 1\}.$

The Fourier extension problem asks for the $L^p\rightarrow L^q$ boundedness of the above operator $E$, i.e. for which pairs of exponents $(p,q)$, we have the following \[\|Ef\|_{L^q(\R^{n+1})}\lesssim \|f\|_{L^p(B_1^n)},\] for $f$ in some suitable dense class.
This question has its origin in \cite{St79}, where Stein considered elliptic surfaces such as the sphere and paraboloid. For these surfaces, the strictly positive Gaussian curvature has an important role. The conjecture for $n=2$ is that $E$ is a bounded operator from $L^p$ to $L^q$ if and only if $q>3$ and $q\ge 2p'$ with $p'$ being the conjugate of $p$, i.e. $1/p+1/p'=1$. 

Let us briefly go through a selective review of the known results. For elliptic surfaces, as a starting point, the Stein-Tomas theorem gives the foundational $L^2$-based restriction estimates \cite{T75,Z74}. Later, Tao-Vargas-Vega as well as Wolff and Tao developed bilinear restriction estimates for the cone and paraboloid \cite{TVV98,W01,T03}. The multilinear theory was developed by Bennett, Carbery and Tao \cite{BCT06} and Bourgain and Guth subsequently introduced the multilinear-to-linear broad–narrow mechanism \cite{BG11}. More recent developments include the Bourgain-Demeter decoupling theorem \cite{BD15}, polynomial methods by Guth \cite{G16,G18}, and the interplay with incidence geometry \cite{WW24}.

The hyperbolic theory is subtler due to the difference in geometry. For example, in the hyperbolic paraboloid defined by \[S:=\{(\xi,\xi_1\xi_2),\xi\in B_1^2\},\] the existence of lines inside the surface complicates the problem. Many elliptic arguments do not transfer directly. 

Early progress was made independently by Vargas and Lee, who established bilinear restriction estimates for the hyperbolic paraboloid and mixed-sign surfaces \cite{V05, L06}. Subsequent improvements were obtained by Cho and Lee to $q>13/4$ and $p=\infty$ \cite{CL17}, later refined by Kim to $q>2p'$ \cite{K17} and by Stovall, who proved the scaling-invariant Fourier restriction estimate for the hyperbolic paraboloid \cite{S18}. More recently, Demeter and Wu proved the $L^p\rightarrow L^p$ boundedness of extension operator for hyperbolic paraboloid in $\R^3$ with $p>22/7$ \cite{DW25}. Another result for global hyperbolic hyperboloid was obtained by Bruce, Oliveira e Silva, and Stovall \cite{BOS21}. Later, Buschenhenke, M\"uller, and Vargas as well as Guo and Oh proved independently improved restriction estimates for surfaces with negative Gaussian curvature in $\R^3$ \cite{BMV23, GO24}. More recently, Guo, Liu, and Xi \cite{GLX26}, further extended the result to the range $q>22/7$ and $q>2p'$.

The study of degenerate surfaces is more complicated since one or more principal curvatures vanish, either at isolated points or along submanifolds. This obstacle causes standard nondegenerate curvature methods to break down.

Firstly, Lee and Vargas proved bilinear estimates for conic type surfaces \cite{LV10}. In dimension three, Ikromov and M\"uller developed a sharp $L^p$ to $L^2$ restriction theory for a large class of smooth surfaces of finite type, including all real-analytic surfaces \cite{IM16}. This work was later partially extended beyond $L^2$ for certain function classes by Buschenhenke, M\"uller, Vargas as well as Schwend and Stovall, in $\R^3$ and $\R^n$ respectively  \cite{BMV17, SS21}.

In \cite{SS21}, the authors estimated the operator norms of the extension operator associated with the perturbed paraboloid over rectangles in terms of the sidelengths of the rectangles, and then applied the corresponding estimates to dyadic pieces of the target degenerate surface. One remarkable feature of the argument is that they allow the error term of the perturbation to have larger higher order derivatives when the eccentricity of rectangles increases. 

The current paper follows the approach in \cite{SS21} to study degenerate hyperbolic surfaces in $\R^3$. We not only generalize the restriction above rectangle result to the hyperbolic case, but also prove it for a wider range, extending the $q$ exponent from $10/3$ to $13/4$.

The majority of the paper will be devoted to showing the restriction over rectangle results for hyperbolic surfaces. The new tool is a bilinear restriction result for the perturbed hyperbolic paraboloid. We extend the result of \cite{O23} to the hyperbolic paraboloid setting, where Oh proved bilinear restriction estimates for the paraboloid. Then we use a bilinear to linear argument inspired by \cite{SS21} to prove the restriction over rectangles results, which in turn, leads to the application, where we present new restriction estimates for a class of hypersurfaces with additive structure.

Before stating the main theorems, let us introduce the following definitions. Let $Q^\ell$ denote an axis-parallel rectangle of dimension $\ell_1\times \ell_2$, and $S$ denote the perturbed hyperbolic paraboloid $$S:=\{\big(\xi_1,\xi_2,g(\xi_1,\xi_2)\big),(\xi_1,\xi_2)\in Q^\ell\},$$ 
where $g(\xi_1,\xi_2):=\xi_1^2-\xi_2^2+h(\xi_1,\xi_2)$ with the error $h$ satisfying $h(0)=0$, $\nabla h(0)=0$, $D^2h(0)=0$ and 
\begin{equation}\label{D2}\|D^2h(\ell_1 \cdot,\ell_2\cdot)\|_{C^N(Q^{\mathbb{1}})}\le \sigma,\;\mathbb{1}=(1,1).
\end{equation}
for some $\sigma$ much smaller than 1, and sufficiently large $N\in\N$. We remark that $\sigma$ will only show up in the implicit constants of the inequalities we will prove later. Also, we will later refer to this condition as $g$ being hyperbolic to order $N$ over $Q^\ell$ with parameter $\sigma$. 

In addition, we remark that this condition is invariant under parabolic rescaling, i.e. for $\tilde{h}(x,y):=r^2 h(r^{-1}x,r^{-1}y)$, $r>0$, we have 
\begin{equation}\label{Rescale}\|D^2\tilde{h}(r\ell_1 \cdot,r\ell_2\cdot)\|_{C^N(Q^{\mathbb{1}})}\le \sigma,\end{equation}
if $h$ obeys \eqref{D2}. Note that $\tilde{h}$ is supported on $Q^{\tilde{\ell}}$ with $\tilde{\ell}_1=r\ell_1$ and $\tilde{\ell}_2=r\ell_2$.

The following theorem characterizes the operator norm of the extension operator associated with a hyperbolic surface over $Q^\ell$ in terms of the side-lengths of $Q^\ell$.
\begin{theorem}[Restriction over axis-parallel rectangles]\label{rectmain}
    Let $\ell\in(0,\infty]^2$ satisfy $\ell_1\le\ell_2$ and let $g$ be hyperbolic to order $N(p,q)$ over $Q^\ell$ with parameter $0<\sigma<\frac{1}{2}$. Also, let $q>p$ satisfy $q> \frac{13}{4}$, and $q=\frac{4-\theta}{2-\theta}p'$ for some $0<\theta\le 1$, then 
    \begin{equation}\label{main1}
        \|\E^\ell_g\|_{L^p\rightarrow L^q}\sim\ell_1^{\theta(\frac{1}{p'}-\frac{1}{q})}.
    \end{equation}
    If $q>p$ satisfy $q>\frac{13}4$ and $q=\frac{3-\theta}{1-\theta}p'$ for some $0\le\theta\le 1$, then 
    \begin{equation}\label{main2}
        \|\E^\ell_g\|_{L^p\rightarrow L^q}\sim(\ell_1\ell_2^\theta)^{(\frac{1}{p'}-\frac{1}{q})}.
    \end{equation}
\end{theorem}
Here, we use subscript $g$ and superscript $\ell$ to emphasize the role of the phase function and sidelengths of the rectangular region. We will drop them later for simplicity when there is no ambiguity.

We remark that $N(p,q)$ is a sufficiently large constant depending only on $p$ and $q$. We require regularity of the defining function depending on $p$ and $q$ because the proof involves the $\epsilon$-removal Lemma [\cite{TV00}, Lemma 2.4]. One particular pair of $(p,q)$ needs a corresponding local estimate with parameter $\epsilon_{p,q}$. In turn, the proof of the local estimate requires the corresponding regularity in the wave packet decomposition. A similar passage was also mentioned in \cite{BMV23}. Please refer to the remark after Theorem 1.3 in \cite{BMV23} for more related details.  

We make another two remarks. First, inequality \eqref{main1} could be generalized to $\theta=0$, i.e. the $q=2p'$ case if the linear extension estimates for $E_g^\mathbb{1}$ were proved on the scaling line $q=2p'$, but currently, it is not known beyond the $L^2$ theory. Secondly,  Our approach in this article can be used to extend the result for elliptic surface of \cite{SS21} to exponents $q>13/4$.

As a corollary, We can also consider hyperbolic surfaces with main term equal to $\xi_1\xi_2$ over rotated rectangles. We will show that  after rescaling, we can assume that the rotated rectangle forms an angle of $\pi/4$ from the $x$-axis. After applying a rotation, we reduce the issue to axis-parallel rectangles with $\xi_1^2-\xi_2^2$ as the main term. 

Due to the two step reduction, the statement will be a little technical. Before stating the result, we need to specify what is means for the hypersurface to be hyperbolic over rotated rectangles.

From symmetry, we only focus on the rectangles whose center lines form a positive slope equal to $\tan\phi$, $0<\phi<\frac{\pi}{2}$. Let $R^{\ell,\phi}$ denote the rectangle of dimension $\ell_1\times\ell_2$ with $\ell_1\le \ell_2$, which forms an angle $\phi$ from the $x$-axis. We define the mappings $L: \R^2\rightarrow \R^2: (\xi_1,\xi_2)\mapsto (\xi_1,(\tan\phi)^{-1}\xi_2)$, and $M: \R^2\rightarrow \R^2: (\xi_1,\xi_2)\mapsto \big(1/2(\xi_1+\xi_2),1/2(-\xi_1+\xi_2)\big)$. A function $g$ is hyperbolic over $R^{\ell,\phi}$ if 
\begin{equation}
    g(\xi_1,\xi_2)=\xi_1\xi_2+h(\xi_1,\xi_2),
\end{equation}
with $h$ satisfying the condition that $\tilde{h}\circ M^{-1}$ satisfies \eqref{D2} over $M\circ L(R^{\ell,\phi})$, and 
$$\tilde{h}(\xi_1,\xi_2):=(\tan\phi)^{-1} h(\xi_1,(\tan\phi)\xi_2).$$ 

Now we are ready to state the result for rotated rectangles.

\begin{theorem}[Restriction to general hyperbolic rectangles]\label{cor}
    Let $\ell\in(0,\infty]^2$ satisfy $\ell_1\le\ell_2$ and $g$ be hyperbolic to order $N(p,q)$ over $Q^{\ell,\phi}$ with parameter $0<\sigma<\frac{1}{2}$. Also, let $q>p$ satisfy $q> \frac{13}{4}$, and $q=\frac{4-\theta}{2-\theta}p'$ for some $0<\theta\le 1$, then 
    \begin{equation}
        \|\E^\ell_g\|_{L^p\rightarrow L^q}\sim(\tan\phi)^{\frac{1}{p'}-\frac{2}{q}}\tilde{\ell_1}^{\theta(\frac{1}{p'}-\frac{1}{q})}.
    \end{equation}
    If $q>p$ satisfy $q>\frac{13}4$ and $q=\frac{3-\theta}{1-\theta}p'$ for some $0\le\theta\le 1$, then 
    \begin{equation}
        \|\E^\ell_g\|_{L^p\rightarrow L^q}\sim(\tan\phi)^{\frac{1}{p'}-\frac{2}{q}}(\tilde{\ell}_1\tilde{\ell}_2^\theta)^{(\frac{1}{p'}-\frac{1}{q})},
    \end{equation} with $\tilde{\ell}= (\min(\ell_1,\cos\phi\cdot\ell_2),\max(\ell_1,\cos\phi\cdot\ell_2))$.   
\end{theorem}
We remark that Theorem \ref{rectmain} can be viewed as a special case of Theorem \ref{cor} with $\phi=\frac{\pi}{4}$ since after rotation, the main term in the defining function for the hypersurface will be transformed from $\xi_1\xi_2$ to $\xi_1^2-\xi_2^2$.

Using H\"older's inequality, Theorem \ref{rectmain} and \ref{cor} imply the following propositions for the pairs $(p,q)$ with $p\ge q>\frac{13}{4}$. 

\begin{prop}\label{omit1}
    Let $\ell\in(0,\infty]^2$ satisfy $\ell_1\le\ell_2$ and $g$ be hyperbolic to order $N(p,q)$ over $Q^\ell$ with parameter $0<\sigma<\frac{1}{2}$. Also, let $p\ge q,\frac{13}{4}<q\le 4$ satisfy $q=\frac{2(3-\theta)}{2-\theta}$, then for all $\epsilon>0$, we have
    \begin{equation}\label{q<p1}
        \|\E^\ell_g\|_{L^p\rightarrow L^q}\lesssim_\epsilon (\ell_1\ell_2)^{\frac{1}{q}-\frac{1}{p}}(\frac{\ell_2}{\ell_1})^\epsilon(\ell_1^\theta)^{1-\frac{2}{q}}.
    \end{equation}
    If $p\ge q>4$, then 
    \begin{equation}\label{q<p2}
        \|\E^\ell_g\|_{L^p\rightarrow L^q}\lesssim \ell_1^{\frac{1}{p'}-\frac{1}{q}}\ell_2^{1-\frac{3}{q}-\frac{1}{p}}.
    \end{equation}
\end{prop}

\begin{prop}\label{omit2}
    Let $\ell\in(0,\infty]^2$ satisfy $\ell_1\le\ell_2$ and $g$ be hyperbolic to order $N(p,q)$  over $Q^{\ell,\phi}$ with parameter $0<\sigma<\frac{1}{2}$. Also, let $p\ge q,\frac{13}{4}<q\le 4$ satisfy $q=\frac{2(3-\theta)}{2-\theta}$, then for all $\epsilon>0$, we have 
    \begin{equation}
        \|\E^\ell_g\|_{L^p\rightarrow L^q}\lesssim_\epsilon(\tan\phi)^{\frac{1}{p'}-\frac{2}{q}}(\tilde{\ell}_1\tilde{\ell}_2)^{\frac{1}{q}-\frac{1}{p}}(\frac{\tilde{\ell}_2}{\tilde{\ell}_1})^\epsilon(\tilde{\ell}_1^\theta)^{1-\frac{2}{q}}.
    \end{equation}
    If $p\ge q>4$, then 
    \begin{equation}
        \|\E^\ell_g\|_{L^p\rightarrow L^q}\lesssim(\tan\phi)^{\frac{1}{p'}-\frac{2}{q}}\tilde{\ell}_1^{\frac{1}{p'}-\frac{1}{q}}\tilde{\ell}_2^{1-\frac{3}{q}-\frac{1}{p}},
    \end{equation} with $\tilde{\ell}= (\min(\ell_1,\cos\phi\cdot\ell_2),\max(\ell_1,\cos\phi\cdot\ell_2))$   
\end{prop}
We emphasize that these two propositions are in the ``$L^p$-worsening" range $q\le p$. The exponent $13/4$ corresponds to the polynomial method range, and 4 is the restriction range for 1D restriction. We also remark that the terms $(\frac{\ell_1}{\ell_2})^\epsilon$ and $(\frac{\tilde{\ell_1}}{\tilde{\ell_2}})^\epsilon$ cannot be completely removed following the same argument as Theorem 1.6 in \cite{SS21}

As an application of Theorem \ref{rectmain}, we prove new extension estimates for the surfaces
\begin{equation}
E_\beta f(t,x):=\int_{[0,1]^2}f(\xi)e^{2\pi i(x\cdot\xi+t g_\beta(\xi))}d\xi,\quad g_\beta:=|\xi_1|^{\beta_1}-|\xi_2|^{\beta_2},\;  \beta_1,\beta_2>1.   
\end{equation}

\begin{prop}\label{deg}
$E_\beta$ can be extended as a bounded operator from $L^p$ to $L^q$ if $q>13/4$, $q>p$, $q>2p'$, and $\frac{q}{p'}\ge\max\big(1+\frac{1}{\frac{1}{\max(\beta_1,\beta_2)}+\frac{1}{2}},1+\frac{1}{\frac{1}{\beta_1}+\frac{1}{\beta_2}}\big)$. Conversely, $E_\beta$ is not a bounded operator in the region $q>13/4$, $q>p$, $q>2p'$, if $\frac{q}{p'}<1+\frac{1}{\frac{1}{\max(\beta_1,\beta_2)}+\frac{1}{2}}$ or $\frac{q}{p'}<1+\frac{1}{\frac{1}{\beta_1}+\frac{1}{\beta_2}}$.   
\end{prop}
We make a similar remark as after Theorem \ref{rectmain} that the above proposition can be generalized to the $(p,q)$ with $\frac{q}{p'}\ge 2$ if the linear extension estimate for $E_g^\ell$ has been proved on the scaling line $q=2p'$.

The above result for elliptic surface in Tao's bilinear range is due to \cite{BMV17}, and the higher dimensional analogue was proved in \cite{SS21}. Also, in a more recent paper by Wang \cite{Wa26}, they proved the local estimate \[\|E_Sg\|_{L^q(B_R^3)}\lesssim_\epsilon R^\epsilon\|g\|_{L^p([0,1]^2)}\] for hyperbolic surfaces given as a difference of two finite type three functions in the range $\frac3{10}\le \frac1q<\frac4{13},$ and $\frac{11}{2q}+\frac34\le \frac5{2p'}.$ Note that our results cover a wider range of exponent pairs.

\subsection{Outline of the paper}
In Section 2, we extend the bilinear restriction result of \cite{O23} to the hyperbolic case. In Section 3, we show how to use the slicing and bilinear to linear reduction methods in \cite{SS21} to get the main theorem, Theorem \ref{rectmain}.  We also provide more detailed computation of the two-step reduction used to prove Proposition \ref{cor}. We will only prove the upper bounds for the operator norm, since the lower bounds follow from the same Knapp example presented in \cite{SS21}, Lemma 3.2. In Section 4, we present the passage from Theorem \ref{rectmain} and \ref{cor} to Propositions \ref{omit1} and \ref{omit2}. In the last section, we present the application of extension estimate for the degenerate surfaces defined in Proposition \ref{deg}.

\subsection{Notation} 
$E$ will always denote the extension operator, sometimes with subscript and superscript to emphasize the surface and dyadic scale of the domain, respectively. For non-negative real numbers $A$ and $B$, $A\lesssim B$ means there exists constant $C$ such that $A\le CB$. We will also sometimes use subscripts for $\lesssim$ to emphasize that the constant depends on the parameters in the subscript. These constants are allowed to change from line to line. $A\sim B$ means $A\lesssim B$ and $B\lesssim A$ simultaneously.
$B\big((a,b),r\big)$ denotes a ball of radius $r$ centered at $(a,b)$. $B_R^d$ denotes the ball in $\R^d$ centered at origin with radius $R$.
$N_r(S)$ means the $r$-neighborhood of the set $S$.

\subsection{Acknowledgment} The author is very thankful to his advisor, Professor Betsy Stovall for introducing the problem, many helpful discussions, and constant support along the project. The author was partially funded by Grant NSF DMS-2246906.

\section{Bilinear restriction for perturbed hyperbolic paraboloid}
In this section, we will extend the bilinear restriction result of \cite{O23} for the paraboloid to perturbed hyperbolic paraboloid over rectangles $Q^\ell$.

Using the invariance of \eqref{D2} under parabolic rescaling (i.e. \eqref{Rescale}), it suffices to prove the bilinear restriction estimate for rectangles with long side equal to 1, and we can get the corresponding results for all other rectangles using parabolic rescaling. For simplicity, from now on, we redefine $Q^\ell$ as the axis-parallel rectangle of dimension $\ell\times 1$, with $\ell\le 1$.
\begin{definition}
We say $f_1$ and $f_2$ satisfy the support separation condition if $f_1$ is supported in the ball $B\big((-1/2,0),1/10\big)\cap Q^\ell$ and $f_2$ is supported in the ball $B\big((1/2,0),1/10\big)\cap Q^\ell$.   
\end{definition}
\begin{theorem}\label{bimain}
    For any $S$ hyperbolic to order $N(p,q)$ over $Q^\ell$ with parameter $\sigma$, any $f_1$ and $f_2$ satisfying the support separation condition, any pair of $(p,q)$ satisfying 
    \begin{equation}
   q>13/4,\quad\frac{5}{q}+\frac{3}{p}<3,\quad \frac{5}{q}+\frac{1}{p}<2,     
    \end{equation}we have
    \begin{equation}
    \||E_Sf_1E_Sf_2|^{1/2}\|_{L^q(\R^3)}\lesssim_{p,q,\sigma}\big(\|f_1\|_{L^p(\R^2)}\|f_2\|_{L^p(\R^2)}\big)^{1/2}.    
    \end{equation}
\end{theorem}
We remark that this estimate has no dependence on $\ell$. This is a non-trivial aspect of the proof since from the definition of hyperbolicity, equation \eqref{D2}, the higher order derivatives of the error term are allowed to get larger as eccentricity of the rectangular region increases. The hyperbolic surface on a smaller piece does not automatically satisfy \eqref{D2} on a larger piece. 

Using interpolation, Theorem \ref{bimain} can be deduced from the cases $q=13/4$ and $p=13/6$ with $\epsilon$ loss, $q=10/3$ and $p=2$ with $\epsilon$ loss, and $q=\infty$ and $p=1$. The $L^\infty$ estimate is H\"older's inequality. The $L^{10/3}$ estimate and $L^{13/4}$ estimate will be  generalizations from \cite{L06} and \cite{O23}, respectively. Since the generalizations we need for the $L^{10/3}$ case are a subset of the case of the $L^{13/4}$ case, we only give the details for the $L^{13/4}$ estimate, which will take up the remainder of the section.

Using the same reduction in \cite{O23} after Proposition 3.1, it suffices to prove the following

\begin{prop}\label{pp}
    For every $\epsilon>0$, and any $f_1$ and $f_2$ satisfying the support separation condition, we have
    \begin{equation}\label{Pp}
    \begin{aligned}
\||Ef_1Ef_2|^{1/2}\|_{L^{13/4}(B_R)}&\le C_\epsilon R^{10\epsilon}\big(\|f_1\|_2^{1/2}\|f_2\|_2^{1/2}\big)^{12/13+\epsilon}\\&\times\big(\max_{\theta:R^{-1/2}-caps}\|f_1\|^{1/2}_{L_{avg}^2(\theta)}\max_{\theta:R^{-1/2}-caps}\|f_2\|^{1/2}_{L_{avg}^2(\theta)}\big)^{1/13-\epsilon}.
    \end{aligned}
    \end{equation}
\end{prop}
Here the notation $``\theta:R^{-1/2}-caps"$ means the family of $R^{-1/2}$-caps covering the unit ball.
The proof of Proposition \ref{pp} uses the polynomial partitioning method. We devote the remainder of the section to it.

\subsection{Proof Outline}
We generalize the approach used in \cite{O23} for the elliptic paraboloid, with extra steps to handle the high eccentricity case $Q^\ell$ with $\ell\ll 1$. Recall that we let $Q^\ell$ denote the rectangle of dimension $\ell\times 1$.

We will first deal with the case where we have the perturbed hyperbolic paraboloid on the unit cube $Q^1$. In this case, every derivative of the error term $h$ is well-controlled, and the adaptation of the proof from \cite{O23} is relatively straightforward. The main reason is that for the wave packet decomposition at scale $R$ defined to be
\begin{align}\label{wpd}
&\T:=\bigcup_{\theta\in R^{-1/2}\Z^2\cap Q^1, v\in R^{1/2}\Z^2\cap B_R^2}T_{\theta,v};\\&
T_{\theta,v}:=\{(x_1,x_2,t)\in B^3_R:|\big(x_1+t\big(2\theta_1+\partial_1 h(\theta_1,\theta_2)\big),x_2+t\big(-2\theta_2+\partial_2 h(\theta_1,\theta_2)\big)\big)-v|\le R^{1/2+\delta}\},\end{align} we can check that the standard estimates of interaction between tubes, i.e. Lemma 2.1 and Lemma 2.2 in \cite{O23} still hold with this family of tubes from the same proof in \cite{G16}, Lemma 2.6, 2.7, and 2.8, or \cite{G18}, Section 3.

For the convenience of the reader, we will sketch the main steps in the argument for $Q^1$, though most of the proof is similar to \cite{O23}.

For the general case of $Q^\ell$ with $\ell\le 1$, a direct application of \cite{O23} will run into problems since the error term $h$ can have large derivatives when $\ell$ is close to 0. This will cause the implicit constants in wave packet decomposition estimates to blow up. Then, the argument in \cite{O23} can only work with large enough $R$, dependent on $\ell$. More precisely, the wave packet decomposition estimates will only work if $R\gg \ell^{-2}$. so that we cannot get an estimate that is independent of $\ell$. 

To deal with this problem, we adapt the method in \cite{SS21} to deal with the smaller scales $R\lesssim \ell^{-2}$. Roughly speaking, under this scenario, the spatial cutoff permits us to blur, so that we can replace the surface with a perturbed hyperbolic paraboloid on the unit cube, and apply the result we derived in the $Q^1$ case. 

We will first present the proof for the unit cube case in the coming subsections, and show how to generalize to $Q^\ell$ in the last subsection.

\subsection{Induction setup}
We will prove that for every $\epsilon>0$, there exists some $R(\epsilon)>1$ so that for all $R_1>2R(\epsilon)$, we have \begin{equation}\label{Ind}\begin{aligned}\||Ef_1Ef_2|^{1/2}\|_{L^{13/4}(B_{R_1})}&\le C_\epsilon R_1^{10\epsilon}\big(\|f_1\|_2^{1/2}\|f_2\|_2^{1/2}\big)^{12/13+\epsilon}\\&\times\big(\max_{\theta:R_1^{-1/2}-caps}\|f_1\|^{1/2}_{L_{avg}^2(\theta)}\max_{\theta:R_1^{-1/2}-caps}\|f_2\|^{1/2}_{L_{avg}^2(\theta)}\big)^{1/13-\epsilon},\end{aligned}\end{equation} provided that equation \eqref{Pp} is true for $R<R_1/2$. Once \eqref{Ind} is shown for $R>R(\epsilon)$, it suffices to prove \eqref{Ind} directly for $R(\epsilon)>R>1$. To this end, \begin{align*}
&\||Ef_1Ef_2|^{1/2}\|_{L^{13/4}(B_{R})}=(\int_{B_R}|Ef_1Ef_2|^{13/8})^{4/13}\lesssim R^{\frac{12}{13}}(\|Ef_1\|_\infty \|Ef_2\|_\infty )^{1/2}\\&\lesssim R^{\frac{12}{13}}(\|f_1\|_1\|f_2\|_1)^{1/2}\lesssim R^{\frac{12}{13}} (\|f_1\|_2\|f_2\|_2)^{1/2}\\&\lesssim R^{\frac{12}{13}} \big(\|f_1\|_2^{1/2}\|f_2\|_2^{1/2}\big)^{12/13+\epsilon}\big(\max_{\theta:R_1^{-1/2}-caps}\|f_1\|^{1/2}_{L_{avg}^2(\theta)}\max_{\theta:R_1^{-1/2}-caps}\|f_2\|^{1/2}_{L_{avg}^2(\theta)}\big)^{1/13-\epsilon}.
\end{align*} The first, second, and last inequality are all triangle inequalities. The second inequality in the second line is H\"older's inequality since $f_1$ and $f_2$ both have finite support. Since $R\lesssim_\epsilon 1$, \eqref{Ind} follows.

\subsection{Proof of Proposition \ref{pp} for $Q^1$}
Let $\delta>0$ be much smaller than $\epsilon$, and $D$ be a large number, independent of $R$. We apply Lemma 2.3 in \cite{O23} to the function $|Ef_1Ef_2|^{13/8}$, and introduce the following definitions.

For polynomials $P_1,P_2,...,P_{3-m}$, $\mathcal{Z}(P_1,P_2,...,P_{3-m})$ denotes the common zero set of $P_i's$. $\mathcal{Z}(P_1,P_2,...,P_{3-m})$ is called an $m$-dimensional transverse complete intersection if the gradients $\nabla P_i(z),i=1,2,...,3-m,$ are linearly independent at all $z\in \mathcal{Z}(P_1,P_2,...,P_{3-m}).$

Using Lemma 2.3 in \cite{O23} with $m=3$, we consider two cases, called the cellular case and the wall case.
\subsubsection{Cellular case}
We say we are in the cellular case if there exists a polynomial $P_1$ of degree at most $D$ so that 
$$\R^3=(\bigsqcup_{k=1}^{\sim D^3} O_k')\bigsqcup \mathcal{Z}(P_1),$$ 
and we have\begin{equation}\label{cell}
    \||Ef_1Ef_2|^{1/2}\|^{13/4}_{L^{13/4}(B_R)}\sim D^3 \||Ef_1Ef_2|^{1/2}\|^{13/4}_{L^{13/4}(O_k)},
\end{equation} for $\sim D^3$ many cells $O_k$, with $O_k:=O_k'\setminus W$ and $W:=N_{R^{1/2+\delta}}(\mathcal{Z}(P_1))$. Then the proof of \eqref{Pp} is the same as in \cite{O23}, page 9, using induction on scales. Indeed, from the wave packet decomposition [\cite{O23}, Lemma 2.1], we have, on each $O_k$,
$$Ef_i=Ef_{i,O_k}+\textrm{RapDec}(R)\|f_i\|_2,\,i=1,2;\;f_{i,O_k}:=\sum_{T\in\T:T\cap O_k\neq \phi}f_{i,T}.$$
We may safely discard the rapid decaying term from now on because the exponent of $R$ in $\textrm{RapDec}(R)$ is sufficiently negative so that the contribution of it in the desired estimate is well-controlled. 

From the orthogonality of wave packets and the fact that each tube will intersect at most $(D+1)$ many cells $O_k$, we can calculate that
\begin{equation}
    \begin{aligned}
        \sum_{O_k}\|f_{i,O_k}\|_2^2&\le C\sum_{O_k}\sum_{T_i\in\T:T_i\cap O_k\neq \phi}\|f_{i,T_i}\|_2^2\\&\le  C\sum_{T_i\in\T}\sum_{O_k:O_k\cap T_i\neq\phi}\|f_{i,T_i}\|_2^2\le 2CD\|f_i\|_2^2.
    \end{aligned}
\end{equation}

We can pick one good cell $O_{k_0}$ using pigeonholing with the following properties $$\|f_{i,O_{k_0}}\|_2^2\lesssim D^{-2}\|f_i\|_2^2,\,i=1,2.$$

Decomposing this cell into finitely many balls of radius $R/2$ and applying induction hypothesis \eqref{Pp} to all these smaller cells, combined with equation \eqref{cell}, we get
\begin{equation}
    \begin{aligned}
\||Ef_1Ef_2|^{1/2}\|_{L^{13/4}(B_R)}&\le C C_\epsilon D^{\frac{12}{13}}R^{10\epsilon}\big(\|f_{1,O_{k_0}}\|_2^{1/2}\|f_{2,O_{k_0}}\|_2^{1/2}\big)^{12/13+\epsilon}\\&\times\big(\max_{\theta:R^{-1/2}-caps}\|f_{1,O_{k_0}}\|^{1/2}_{L_{avg}^2(\theta)}\max_{\theta:R^{-1/2}-caps}\|f_{2,O_{k_0}}\|^{1/2}_{L_{avg}^2(\theta)}\big)^{1/13-\epsilon}\\&\le C D^{-\epsilon} C_\epsilon R^{10\epsilon}\big(\|f_1\|_2^{1/2}\|f_2\|_2^{1/2}\big)^{12/13+\epsilon}\\&\times\big(\max_{\theta:R^{-1/2}-caps}\|f_1\|^{1/2}_{L_{avg}^2(\theta)}\max_{\theta:R^{-1/2}-caps}\|f_2\|^{1/2}_{L_{avg}^2(\theta)}\big)^{1/13-\epsilon}.        
    \end{aligned}
\end{equation}
Finally, we note that we can conclude the induction if we choose $D$ large enough so that $C D^{-\epsilon}\le 1$.
\subsubsection{Wall case}
If we are not in the cellular case, there exists a two-dimensional transverse complete intersection $\mathcal{Z}(P_1)$ of degree at most $D$ such that
\begin{equation}
    \||Ef_1Ef_2|^{1/2}\|^{13/4}_{L^{13/4}(B_R)}\lesssim \||Ef_1Ef_2|^{1/2}\|^{13/4}_{L^{13/4}(B_R\bigcap N_{R^{1/2+\delta}}(\mathcal{Z}(P_1))}.
\end{equation} Let $D_1$ be a large constant compared with $D$, independent of $R$.

We first introduce the following lemma. Recall that $\delta$ is a parameter much smaller than $\epsilon$.
\begin{lemma}\label{1d}
For any one-dimensional transverse complete intersection $\mathcal{Z}(P_1,P_2)$ of degree at most $D_1$ and $f_1$, $f_2$ satisfying the support separation condition, we have
\begin{equation}\begin{aligned}
    &\||Ef_1Ef_2|^{1/2}\|^{13/4}_{L^{13/4}(B_R\bigcap N_{10R^{1/2+\delta}}(\mathcal{Z}(P_1,P_2))}\\&\lesssim C_\epsilon R^{-c\delta\epsilon}R^{10\epsilon}\big(\|f_1\|_2^{1/2}\|f_2\|_2^{1/2}\big)^{12/13+\epsilon}\big(\max_{\theta:R_1^{-1/2}-caps}\|f_1\|^{1/2}_{L_{avg}^2(\theta)}\max_{\theta:R_1^{-1/2}-caps}\|f_2\|^{1/2}_{L_{avg}^2(\theta)}\big)^{1/13-\epsilon},
\end{aligned}\end{equation} under the condition that \eqref{Pp} is true for all the radii less than $R/2$.
\end{lemma}
The proof of the Lemma is postponed until the end of this section. With this Lemma, we may assume that $|Ef_1Ef_2|^{1/2}$ does not concentrate on any one dimensional transverse complete intersection. This will further imply that the weight will mostly concentrate on the regular balls, which are defined below.
\begin{definition}
We say that a ball $B(x_0,R^{1/2+\delta})\subset N_{R^{1/2+\delta}(\mathcal{Z}(P_1))}\bigcap B_R$ is regular if on each connected component of $\mathcal{Z}(P_1)\bigcap B(x_0,10R^{1/2+\delta})$, the tangent planes $T_z\mathcal{Z}(P_1)$ are within distance $1/100$ from each other across  $z\in\mathcal{Z}(P_1)\bigcap B(x_0,10R^{1/2+\delta})$ measured in Grassmannians.
\end{definition}
Let $\mathcal{Z}_w:=\{x\in\mathcal{Z}(P_1):\nabla P_1(x)\wedge w=0\}$. It is proved in \cite{G18}, page 125-126 that for generic $w\in\Lambda^2\R^3$, $\mathcal{Z}_w\subset \mathcal{Z}(P_1)$ is a transverse complete intersection of dimension 1, defined using polynomials of degree $\lesssim D$. We can choose a set of $\lesssim1$ many $w\in\Lambda^2\R^3$ so that on each connected component of $\mathcal{Z}(P_1)\setminus \bigcup_w\mathcal{Z}_w$, the tangent planes $T_z\mathcal{Z}(P_1)$ are within distance $1/100$ from each other across  $z\in\mathcal{Z}(P_1)\bigcap B(x_0,10R^{1/2+\delta})$ measured in Grassmannians.

Thus, each ball $B(x_0,R^{1/2+\delta})\subset N_{R^{1/2+\delta}(\mathcal{Z}(P_1))}\bigcap B_R$ that does not intersect $\bigcup_w N_{10R^{1/2+\delta}}(\mathcal{Z}_w)$ is regular, and since $|Ef_1Ef_2|^{1/2}$ concentrates away from any one dimensional transverse complete intersection, we get that $|Ef_1Ef_2|^{1/2}$ is concentrated on regular balls. For each regular ball $B$, we pick a point $z\in B\bigcap\mathcal{Z}(P_1)$ and take $V_B$ as the 2 dimensional tangent plane $T_z\mathcal{Z}(P_1)$. After pigeonholing, we can get
\begin{equation}
\||Ef_1Ef_2|^{1/2}\|^{13/4}_{L^{13/4}(B_R\bigcap N_{R^{1/2+\delta}}(\mathcal{Z}(P_1))}\lesssim \||Ef_1Ef_2|^{1/2}\|^{13/4}_{L^{13/4}(\bigcup_{B\in\mathcal{B}_V}B)},
\end{equation} where $\mathcal{B}_V$ consists of those regular balls where the angle between $V_B$ and $V$ is smaller than $\gamma:=\frac{1}{100}$.

We introduce the following notations
\begin{equation}
    \begin{aligned}
&N_1:=\bigcup_{B\in\mathcal{B}_V}B\subset \bigg(N_{R^{1/2+\delta}(\mathcal{Z}(P_1))}\bigcap B_R\bigg),\\&\T_{\ge 4\gamma}:=\{T\in\T:Angle(v(T),V)\ge4\gamma\},\\& \T_{< 4\gamma}:=\{T\in\T:Angle(v(T),V)<4\gamma\},\\&
f_{i,\ge4\gamma}:=\sum_{T\in\T_{\ge4\gamma}}f_{i,T},\quad\quad f_{i,<4\gamma}:=\sum_{T\in\T_{<4\gamma}}f_{i,T},\;i=1,2.
    \end{aligned}
\end{equation}
Thus, 
\begin{equation}\label{angle}
    \begin{aligned}
    \||Ef_1Ef_2|^{1/2}\|^{13/4}_{L^{13/4}(\bigcup_{B\in\mathcal{B}_V}B)}\lesssim &\||Ef_{1,\ge4\gamma}Ef_{2,\ge4\gamma}|^{1/2}\|^{13/4}_{L^{13/4}(N_1)}\\&+\||Ef_{1,<4\gamma}Ef_{2,\ge4\gamma}|^{1/2}\|^{13/4}_{L^{13/4}(N_1)}\\&+\||Ef_{1,\ge4\gamma}Ef_{2,<4\gamma}|^{1/2}\|^{13/4}_{L^{13/4}(N_1)}\\&+\||Ef_{1,<4\gamma}Ef_{2,<4\gamma}|^{1/2}\|^{13/4}_{L^{13/4}(N_1)}.
    \end{aligned}
\end{equation}
\paragraph{High angle case}
If the first three terms dominate, the desired estimate \eqref{Pp} will follow from \begin{align}\||Ef_{1,\ge 4\gamma}Ef_2|^{1/2}\|_{L^{13/4}(N_1)}&\le C_\epsilon R^{10\epsilon}(\|f_1\|_2^{1/2}\|f_2\|_2^{1/2})^{12/13+\epsilon}\\&(\max_{\theta}\|f_1\|_{L_{avg}^2(\theta)}^{1/2}\max_{\theta}\|f_2\|_{L_{avg}^2(\theta)}^{1/2})^{1/13-\epsilon}.
\end{align}
We apply the polynomial partitioning (Lemma 2.3 in \cite{O23}) to the function $\chi_{N_1}|Ef_{1.\ge 4\gamma}Ef_2|^{1/2}$ with degree $D_1$ to be determined later. Due to Lemma \ref{1d}, it suffices to only consider the cellular case. Thus, there exists a polynomial $P_2:\,\R^3\rightarrow \R$ of degree at most $D_1$ such that we have the following decomposition
\[\R^3\setminus \mathcal{Z}(P_2)=\bigsqcup_{k=1}^{\sim D_1^2}\tilde{O_k'}.\]
Also, most of the weight of $\chi_{N_1}|Ef_{1.\ge 4\gamma}Ef_2|^{1/2}$ is essentially equally distributed over $\sim D_1^2$ many cells $\tilde{O_k}$:
\begin{equation}\label{highd}\||Ef_{1,\ge 4\gamma}Ef_2|^{1/2}\|_{L^{13/4}(N_1)}^{13/4}\sim D_1^2\||Ef_{1,\ge 4\gamma}Ef_2|^{1/2}\|_{L^{13/4}(N_1\cap \tilde{O_k})},\end{equation} with $\tilde{O_k}:=B_R\cap(\tilde{O_k'}\setminus N_{R^{1/2+\delta}(\mathcal{Z}(P_2))}).$

For the family of tubes $\T$ defined in \eqref{wpd}, let $\T_k$ denote the sub-family of tubes intersecting $\tilde{O_k}$. Since each tube in $\T$ can pass through at most $D_1+1$ many cells $\tilde{O_k}$, from almost orthogonality, we get
\begin{equation}\label{hight}
\sum_k\|\sum_{T\in\T_k} f_{2,T}\|_{L^2}^2\lesssim D_1\|f_2\|_2^2.   
\end{equation}

The key observation that makes the high angle case more tractable is that each tube $T$ in $\T_{\ge 4 \gamma}$ can intersect different $\tilde{O_k}\cap N_1$ for at most $O(D^3)$ many $k$'s, since the set $T\bigcap\mathcal{Z}(P_1)$ is contained in at most $O(D^3)$ balls of radius $R^{1/2+\delta}$, from Lemma 5.7 in \cite{G18}. With this observation and almost orthogonality, we get
\begin{equation}\label{high4g}
\sum_k\|\sum_{T\in\T_{\ge 4\gamma,k}} f_{1,T}\|_{L^2}^2\lesssim D^3\|f_1\|_2^2.     
\end{equation}

With \eqref{hight} and \eqref{high4g}, we can pick one good cell $\tilde{O_{k_0}}$ using pigeonholing with the following properties $$\|\sum_{T\in\T_{k_0}} f_{2,T}\|_{L^2}^2\lesssim D_1^{-1}\|f_2\|_2^2\quad;\quad\|\sum_{T\in\T_{\ge 4\gamma,k_0}} f_{1,T}\|_{L^2}^2\lesssim D^3D_1^{-2}\|f_1\|_2^2.$$
Decomposing this cell into finitely many balls of radius $R/2$ and applying our induction hypothesis \eqref{Pp} to all these smaller cells, combined with equation \eqref{highd}, after some arithmetic, we get
\begin{equation}
    \begin{aligned}
\||Ef_1Ef_2|^{1/2}\|_{L^{13/4}(N_1)}&\lesssim D^{\frac{9}{13}+\epsilon}D_1^{-\frac{1}{13}}C_\epsilon R^{10\epsilon}\big(\|f_1\|_2^{1/2}\|f_2\|_2^{1/2}\big)^{12/13+\epsilon}\\&\times \big(\max_{\theta:R^{-1/2}-caps}\|f_1\|^{1/2}_{L_{avg}^2(\theta)}\max_{\theta:R^{-1/2}-caps}\|f_2\|^{1/2}_{L_{avg}^2(\theta)}\big)^{1/13-\epsilon}.        
    \end{aligned}
\end{equation}
It suffices to choose $D_1$ much larger than $D$ to conclude the induction.

\paragraph{Low angle case}
If the last term in \eqref{angle} dominates, we are in the low angle case. We do a further decomposition into transversal and tangential tubes. First let $W:=B_R\bigcap N_{R^{1/2+\delta}(\mathcal{Z}(P_1))}$ and cover $B_R$ by balls $B_j$ of radius $R^{1-\delta}.$ Recall that $\delta$ is a parameter much smaller than $\epsilon$. For each $B_j$, we make the following definition.
\begin{definition}
    The set of tangential tubes $\T_{j,-}$ is the set of all $T\in\T_{<4\gamma}$ obeying
    \begin{itemize}
        \item $T\bigcap W\bigcap B_j\neq\emptyset$,
        \item For any $z\in\mathcal{Z}(P_1)\bigcap 2B_j\bigcap 10T$, the angle between $v(T)$ and $T_z\mathcal{Z}(P_1)$ is $\le R^{-1/2+2\delta}.$
    \end{itemize}
\end{definition}
The rest of the tubes in $\T_{<4\gamma}$ are transversal tubes and we denote the set of tangential tubes by $\T_{j,+}$. We denote \[f_{i,j,-}:=\sum_{T\in\T_{j,-}}f_{i,T},\quad f_{i,j,+}:=\sum_{T\in\T_{j,+}}f_{i,T},\;i=1,2.\]
We can thus decompose the low angle term further into transversal and tangential terms.
\begin{equation}
    \begin{aligned}
\||Ef_{1,<4\gamma}Ef_{2,<4\gamma}|^{1/2}\|^{13/4}_{L^{13/4}(N_1)}&\lesssim\sum_{B_j}\||Ef_{1,j,+}Ef_{2,j,+}|^{1/2}\|^{13/4}_{L^{13/4}(N_1\bigcap B_j)}\\& +\sum_{B_j}\||Ef_{1,j,-}Ef_{2,j,+}|^{1/2}\|^{13/4}_{L^{13/4}(N_1\bigcap B_j)} \\& +\sum_{B_j}\||Ef_{1,j,+}Ef_{2,j,-}|^{1/2}\|^{13/4}_{L^{13/4}(N_1\bigcap B_j)} \\&+ \sum_{B_j}\||Ef_{1,j,-}Ef_{2,j,-}|^{1/2}\|^{13/4}_{L^{13/4}(N_1\bigcap B_j)}.
    \end{aligned}
\end{equation}

If the first sum dominates, we are in the transversal wall case and we can proceed by applying our induction hypothesis to each term as in \cite{O23}, page 14. The key ingredient for us to be able to sum in $j$ is that each tube $T\in\T$ belongs to $\T_{j,+}$ for at most $O(D^3)$ many $j$, from Lemma 5.7 of \cite{G18}, which does not distinguish between elliptic or hyperbolic surfaces.

For the 3 other terms, it suffices to prove
\begin{equation}\label{tang1}
    \begin{aligned}
&\||Eg_{1,-}Eg_{2,<4\gamma}|^{1/2}\|_{L^{13/4}(N_1\bigcap B_j)}\\&\le C_\epsilon R^{O(\delta)}\big(\|g_1\|_2^{1/2}\|g_2\|_2^{1/2}\big)^{12/13}\big(\max_{\theta:R^{-1/2}-caps}\|g_1\|^{1/2}_{L_{avg}^2(\theta)}\max_{\theta:R^{-1/2}-caps}\|g_2\|^{1/2}_{L_{avg}^2(\theta)}\big)^{1/13}.
    \end{aligned}
\end{equation}
The lefthand side of \eqref{tang1} can be bounded using H\"older by \begin{equation}\label{L2L4}
\||Eg_{1,j,-}Eg_{2,<4\gamma}|^{\frac{1}{2}}\|_{L^{13/4}(B_j\bigcap N_1)}\le \||Eg_{1,j,-}Eg_{2,<4\gamma}|^{\frac{1}{2}}\|_{L^2(B_j\bigcap N_1)}^{3/13} \||Eg_{1,j,-}Eg_{2,<4\gamma}|^{\frac{1}{2}}\|_{L^4(B_j\bigcap N_1)}^{10/13}.
\end{equation}
Thus, it remains to derive the $L^2$ and $L^4$ estimates separately.

The first step is to do dyadic pigeonholing on wave packets $T\in\T_{2,<4\gamma}$ of $Ef_{2,<4\gamma}$ according to the length of $T\bigcap N_{R^{1/2+\delta}(\mathcal{Z}(P_1))}.$

We start with using the observation and reduction on page 15 of \cite{O23}: For every tube $T\in\T$, there exist subtubes $\bar{T}_{T,m}$ of dimension $5R^{\frac{1}{2}+\delta}\times 5R^{\frac{1}{2}+\delta}\times c_m$ with $c_m\ge R^{\frac{1}{2}+\delta}$ such that \begin{equation}
    T\bigcap N_{R^{\frac{1}{2}+\delta}}(Z(P_1))\subset\bigsqcup_{m=1}^{\lesssim C_D}\bar{T}_{T,m}\subset N_{20R^{\frac{1}{2}+\delta}}(Z(P_1)),\quad \textrm{dist}(\bar{T}_{T,m},\bar{T}_{T,m'})\ge 2 R^{\frac{1}{2}+\delta},\forall m,m'.
\end{equation} Thus, for $x\in B_j\bigcap N_1$, we have \begin{equation}\label{subtube}
    Eg_{2,<4\gamma}(x)=\sum_{T\in\T_{<4\gamma}}\sum_{m=1}^{\lesssim C_D}\chi_{\bar{T}_{T,m}}(x)Eg_{2,T}(x)+RapDec(R)\|g_2\|_{L^2}.
\end{equation}

We then apply dyadic pigeonholing to the tubes $T\in\T_{<4\gamma}$ by the length of $c$, which is the length of the longest subtube with dimension $R^{\frac{1}{2}+\delta}\times R^{\frac{1}{2}+\delta}\times c$ in $T\bigcap N_{20R^{\frac{1}{2}+\delta}}(Z(P_1))$. Since $R^{\frac{1}{2}+\delta}\le c\le R$, with a $O(\log(R))$ loss from pigeonholing, we may assume that for all the tubes in $\T_{<4\gamma}$, the longest length of all possible subtubes contained in $T\bigcap N_{20R^{\frac{1}{2}+\delta}}(Z(P_1))$ is comparable to $c$ for some fixed $c$.

Let $\Theta_{Leng(c)}$ denote the collection of directions of the tubes $T\in\T$ whose intersection with $W\bigcap B_j$ contains a tube of dimension $R^{1/2+\delta}\times R^{1/2+\delta}\times c.$ Define
\begin{equation}
    \T_{Leng(c)}:=\{T\in\T_{<4\gamma}:v(T)\in \Theta_{Leng(c)}\}\quad f_{2,Leng(c)}:=\sum_{T\in \T_{Leng(c)}} f_{2,T}.
\end{equation}

For the $L^2$ estimate, we first use Cauchy-Schwarz to get \begin{equation}\label{l21}
\begin{aligned}
    \||Eg_{1,j,-}Eg_{2,<4\gamma}|^{\frac{1}{2}}\|_{L^2(B_j\bigcap N_1)}\le\|Eg_{1,j,-}\|_{L^2(B_j)}^{\frac{1}{2}}\|Eg_{2,<4\gamma}\|_{L^2(B_j\bigcap W)}^{\frac{1}{2}}\;.
\end{aligned}
\end{equation}
For $Eg_{1,j,-}$, we just use the standard $L^2$-estimate (Fubini's Theorem) to get that \begin{equation}\label{l22}
   \|Eg_{1,j,-}\|_{L^2(B_j)}^2\le R^{(1-\delta)}\big(\sum_{T\in\T_{j,-}}\|g_{1,T}\|_2^2\big)\;.
\end{equation}
For $Eg_{2,<4\gamma}$, we use the expression \eqref{subtube} to get that \begin{equation}
 \|Eg_{2,<4\gamma}\|_{L^2(B_j)}^2\lesssim  \|\sum_{T\in\T_{<4\gamma}}\sum_{m=1}^{\lesssim C_D}\chi_{\bar{T}_{T,m}}(x)Eg_{2,T}(x)\|_{L^2(B_j\bigcap W)}^2\lesssim \|\sum_{T\in\T_{<4\gamma}}\chi_{\bar{T}_{T}}(x)Eg_{2,T}(x)\|_{L^2(B_j\bigcap W)}^2.
\end{equation} Here $C_D$ is a constant independent of $R$. The last inequality follows since we used triangle inequality to replace the summation in $m$ by one single term with $m_{T,0}$. For simplicity, we denote $\bar{T}_{T,m_{T,0}}$ as $\bar{T}_{T}$. 

Now we can apply the Lemma 4.5 (Local constancy) in \cite{O23} to get that \begin{equation}\label{l23}
\|Eg_{2,<4\gamma}\|_{L^2(B_j\bigcap W)}^2\lesssim R^{O(\delta)}[\frac{c}{R}\sum_{T\in\T_{Leng(c)}}\|Eg_{2,T}\|_{L^2(w_{B_j})}^2]\le R^{O(\delta)}[\frac{c}{R}R^{(1-\delta)}\sum_{T\in\T_{Leng(c)}}\|g_{2,T}\|^2_{L^2}]^.  
\end{equation}
Then, we combine \eqref{l21}, \eqref{l22}, and \eqref{l23} to get the final $L^2$-estimate as \begin{equation}\label{L2}
 \||Eg_{1,j,-}Eg_{2,<4\gamma}|^{\frac{1}{2}}\|_{L^2(B_j)}\lesssim R^{O(\delta)}\big(\frac{c}{R}\big)^{\frac{1}{4}}R^{\frac{1}{2}}\big(\sum_{T\in\T_{j,-}}\|g_{1,T}\|_2^2\big)^{\frac{1}{4}}\big(\sum_{T\in\T_{Leng(c)}}\|g_{2,T}\|^2_{L^2}\big)^{\frac{1}{4}}.
\end{equation}

Next, we turn to the $L^4$ estimate. Note that $B_j\bigcap N_1$ is a union of regular balls $Q$ of radius $R^{1/2}$. Let
\begin{equation}
\begin{aligned}
    \T_{Leng(c),Q}:=\{T\in \T_{Leng(c)}:T\bigcap Q\bigcap B_j\bigcap W\neq\emptyset\}.\\
    \T_{j,-,Q}:=\{T\in \T_{j,-}:T\bigcap Q\bigcap B_j\bigcap W\neq\emptyset\}.
    \end{aligned}
\end{equation}
By dyadic pigeonholing on $c$ and triangle inequality, we have
\begin{equation}
    \begin{aligned}
 &\||Eg_{1,j,-}Eg_{2,<4\gamma}|^{\frac{1}{2}}\|_{L^4(B_j\bigcap N_1)}\\&\lesssim\big(\sum_{Q:Q\cap W\cap B_j\neq\phi}\||\sum_{T\in\T_{j,-,Q}}Eg_{1,T}\sum_{T\in\T_{Leng(c),Q}}Eg_{2,T}|^{\frac{1}{2}}\|^4_{L^4(Q)}\big)^{\frac{1}{4}}.      
    \end{aligned}
\end{equation}

The goal is to prove the following estimate
\begin{equation}
    \begin{aligned}
\||\sum_{T\in\T_{j,-,Q}}Eg_{1,T}\sum_{T\in\T_{Leng(c),Q}}Eg_{2,T}|^{\frac{1}{2}}\|^4_{L^4(Q)}\lesssim R^{-\frac{1}{2}+O(\delta)}\big(\sum_{T\in \T_{j,-,Q}}\|g_{1,T}\|_2^2\big)\big(\sum_{T\in \T_{Leng(c),Q}}\|g_{2,T}\|_2^2\big),       
    \end{aligned}
\end{equation}
which will be implied by the following Lemma in \cite{B22} with the one-dimensional submanifold formed by the support of $\sum_{T\in\T_{j,-,Q}}g_{1,T}$.
\begin{lemma}\label{be}
Let $S_i'\subset S_i$ be submanifolds of codimension $d_i$, $i=1,2$ with the property that there exists $\nu>0$ such that
    \begin{equation}\label{tran}
        |N_{\zeta_1}S_1'\wedge N_{\zeta_2}S_2'|\ge\nu,
    \end{equation}
    for all choices of $\zeta_i\in S_i'$. Assume that we have functions $g_i\in L^2(S_i)$ with $supp(g_i)\subset N_{\mu_i}(S_i')\bigcap S_i$, then
    \begin{equation}
        \|\Pi_{i=1}^2E_{S_i}g_i\|_{L^{2}(B_R)}\lesssim C_\epsilon\Pi_{i=1}^2 \mu_i^{\frac{d_i}{2}}R^\epsilon\Pi_{i=1}^2\|g_i\|_{L^2(S_i)}.
    \end{equation}
\end{lemma}
In order to apply the Lemma, we need to check the transversality condition \eqref{tran}.

By definition, any $T\in\T_{j,-,Q}$ forms angle $<R^{-1/2+2\delta}$ with the tangent plane $T_z\mathcal{Z}(P_1)$ for some $z$. Because $Q$ is an $R^{1/2}$-cube, we can choose a common $z\in Q$, so that all the tubes $T\in\T_{j,-,Q}$ form an angle $<R^{-1/2+2\delta}$ with some fixed plane $P$. 

Also, since $\T_{j,-,Q},\T_{Leng(c),Q}\subset \T_{<4\gamma},$ all the other tubes form an angle $\le O(\gamma)$, which is much smaller than the separation, from the same plane $P$.

Suppose that the normal vector of $P$ is given by $(a,b,c)$ with $|(a,b,c)|=1$. Then, the tubes in $\T_1:=\T_{j,-,Q}$ and $\T_2:=\T_{j,Leng(c),Q}$ should live in the region on the surface whose projection to $Q^\ell$ lives in the region $R_1,R_2$, respectively,

$$R_1:=\{(x,y)\in Q^\ell:-2xa-\partial_1h(x,y)a+2yb-\partial_2h(x,y)b+c=O(R^{-1/2+2\delta})\};$$

$$R_2:=\{(x,y)\in Q^\ell:-2xa-\partial_1h(x,y)a+2yb-\partial_2h(x,y)b+c=O(\gamma)\}.$$

Note that in order for $R_1$ to intersect $U_1:=B((-1/2,0),1/10)$ and $R_2$ to intersect $U_2:=B((1/2,0),1/10)$, we need $|b|\ge \frac{1}{2}$ because the contributing terms will come from $2xa$ and $c$ in the defining equations for $R$ if $|b|$ is small. However, near $U_1$ and $U_2$, the sign of $x$ is flipped in the first term, so $c-2xa$ cannot be small on both $U_1$ and $U_2$. 

Thus, we see that $R_1$ is a $(CR^{-1/2+2\delta})$-neighborhood of the curve $(x,\sigma(x))$ defined implicitly from the equation $$-2xa-\partial_1h(x,y)a+2yb-\partial_2h(x,y)b+c=0.$$ Using the implicit differentiation theorem, we can show $|\sigma'(x)|\lesssim 1$. Thus, the first submanifold $S_1$ as in the Lemma \ref{be} is the graph of the perturbed hyperbolic paraboloid $S$ over curve $(x,\sigma(x))$, which can be written as $N_{CR^{-1/2+2\delta}}r$ for the curve $r(x):=(x,\sigma(x),x^2-\sigma(x)^2+h(x,\sigma(x))),$ with $|\sigma'(x)|\lesssim 1$.

We can calculate that the tangent vector of the curve is 
$$r'(x)=(1,\sigma'(x),2x-2\sigma(x)\sigma'(x)+\partial_1 h(x,\sigma(x))+\partial_2 h(x,\sigma(x))\sigma'(x)).$$

On the other hand, the tubes in $\T_2$ live near $(1,0,0)$, so that the normal vector at point $(x',y')$ can be given by $$n(x',y')=(-2x'-\partial_1h(x',y'),2y'-\partial_2h(x',y'),1).$$

Taking inner product of $r'(x)$ and $n(x',y')$, and noticing that all the terms other than $-2x'+2x$ will be small, since they either involve the first derivatives of $h$, which is small since $D^2h$ is small, or involve $y'$ and $\sigma(x)=y$, which are both small since we chose our balls to be small enough in the $y$ direction in the beginning. Since $x\sim -1$ and $x'\sim1$, $r'$ and $n$ are transversal. Therefore, the transversality condition in Lemma \ref{be} is satisfied, and we can apply the Lemma to get the estimate on one cube $Q$. It remains to sum in $Q$. 

Note that because the tubes in $\T_{j,-,Q}$ and $\T_{Leng(c),Q}$ intersect transversally, there are at most $R^{O(\delta)}$ many $Q$ intersecting both of a fixed pair $(T,T')\in\T_{j,-,Q}\times \T_{Leng(c),Q}$. Thus, we can sum in $Q$ and conclude that 
\begin{equation}\label{L4}
    \begin{aligned}
&\||Eg_{1,j,-}Eg_{2,<4\gamma}|^{\frac{1}{2}}\|^4_{L^4(B_j\bigcap N_1)}\\&\lesssim\sum_{Q:Q\cap W\cap B_j\neq\phi}\||\sum_{T\in\T_{j,-,Q}}Eg_{1,T}\sum_{T\in\T_{Leng(c),Q}}Eg_{2,T}|^{\frac{1}{2}}\|^4_{L^4(Q)}\\&\lesssim R^{-1/2+O(\delta)} \sum_{T\in \T_{j,-}}\|g_{1,T}\|_2^2\sum_{T\in \T_{Leng(c)}}\|g_{2,T}\|_2^2.     
    \end{aligned}
\end{equation}

Note that each tube $T\in \T_{j,-}$ provides a line segment of length $R^{1-O(\delta)}$ in $N_{R^{1/2}}\mathcal{Z}(P_1)$ and each tube $T\in \T_{Leng(c)}$ provides a line segment of length $c$ in $N_{R^{1/2}}\mathcal{Z}(P_1)$. Using the polynomial Wolff axioms (Lemma 2.6 in \cite{O23}), we can get an upper bound on the number of directions for the tubes, i.e. the number of caps $\theta$ in $\T_{j,-}$ and $\T_{Leng(c)}$. Using triangle inequality, we get
\begin{equation}\label{polywolff}
    \begin{aligned}
\sum_{T\in \T_{j,-}}\|g_{1,T}\|_2^2\lesssim R^{-1/2+O(\delta)} \max_{\theta}\|g_1\|^2_{L^2_{avg}(\theta)}\\ \sum_{T\in \T_{Leng(c)}}\|g_{2,T}\|_2^2\lesssim c^{-1}R^{1/2+O(\delta)} \max_{\theta}\|g_2\|^2_{L^2_{avg}(\theta)}.
    \end{aligned}
\end{equation}
Now, we combine equation \eqref{L2L4}, the $L^2$ estimate \eqref{L2}, the $L^4$ estimate \eqref{L4}, and equation \eqref{polywolff} to get
\begin{equation}
    \begin{aligned}
&\||Eg_{1,-}Eg_{2,<4\gamma}|^{1/2}\|_{L^{13/4}(N_1\bigcap B_j)}\\&\lesssim \bigg(R^{O(\delta)}\big(\frac{c}{R}\big)^{\frac{1}{4}}R^{\frac{1}{2}}\big(\sum_{T\in\T_{j,-}}\|g_{1,T}\|_2^2\big)^{\frac{1}{4}}\big(\sum_{T\in\T_{Leng(c)}}\|g_{2,T}\|^2_{L^2}\big)^{\frac{1}{4}}\bigg)^{3/13}\bigg(R^{-1/2+O(\delta)} \sum_{T\in \T_{j,-}}\|g_{1,T}\|_2^2\sum_{T\in \T_{Leng(c)}}\|g_{2,T}\|_2^2\bigg)^{5/26}
\\&= R^{O(\delta)}c^{\frac{3}{52}}R^{-\frac{1}{26}}(\sum_{T\in \T_{j,-}}\|g_{1,T}\|_2^2)^{3/13}(\sum_{T\in \T_{j,-}}\|g_{1,T}\|_2^2)^{1/52}(\sum_{T\in \T_{Leng(c)}}\|g_{2,T}\|_2^2)^{3/13}(\sum_{T\in \T_{Leng(c)}}\|g_{2,T}\|_2^2)^{1/52}\\&\lesssim
R^{O(\delta)}c^{\frac{3}{52}}R^{-\frac{1}{26}} (\|g_1\|_2^2\|g_2\|_2^2)^{3/13}(R^{-1/2+O(\delta)} \max_{\theta}\|g_1\|^2_{L^2_{avg}(\theta)}c^{-1}R^{1/2+O(\delta)} \max_{\theta}\|g_2\|^2_{L^2_{avg}(\theta)})^{1/52}
\\&\lesssim R^{-\frac{1}{26}+O(\delta)}c^{\frac{1}{26}}\big(\|g_1\|_2^{1/2}\|g_2\|_2^{1/2}\big)^{12/13}\big(\max_{\theta}\|g_1\|^{1/2}_{L_{avg}^2(\theta)}\max_{\theta}\|g_2\|^{1/2}_{L_{avg}^2(\theta)}\big)^{1/13}.        
    \end{aligned}
\end{equation}
This implies the desired estimate since $c\le R$.

\subsubsection{Proof of Lemma \ref{1d}}
It remains to prove
\begin{lemma}\label{bilinear}
For any one dimensional transverse complete intersection $\mathcal{Z}(P_1,P_2)$ of degree at most $D_1$, we have
\begin{equation}\begin{aligned}
    &\||Ef_1Ef_2|^{1/2}\|^{13/4}_{L^{13/4}(B_R\bigcap N_{R^{1/2+\delta}}(\mathcal{Z}(P_1,P_2))}\\&\lesssim C_\epsilon R^{-c\delta\epsilon}R^{10\epsilon}\big(\|f_1\|_2^{1/2}\|f_2\|_2^{1/2}\big)^{12/13+\epsilon}\big(\max_{\theta:R^{-1/2}-caps}\|f_1\|^{1/2}_{L_{avg}^2(\theta)}\max_{\theta:R^{-1/2}-caps}\|f_2\|^{1/2}_{L_{avg}^2(\theta)}\big)^{1/13-\epsilon},
\end{aligned}\end{equation} under the condition that \eqref{Pp} is true for all the radii less than $R/2$.
\end{lemma}
The proof is essentially the same as the low angle case, where we divide the tubes into tangential tubes and transversal tubes. For the transversal term, we can induct on scales, and for the tangential term, we can apply the polynomial Wolff axiom with codimension equal to 2, yielding a better decay
\begin{equation}
    \sum_{T\in \T_{j,-}}\|g_{1,T}\|_2^2\lesssim R^{-1+O(\delta)} \max_{\theta}\|g_1\|^2_{L^2_{avg}(\theta)}.
\end{equation}
Please refer to \cite{O23}, Section 5 for a detailed proof. 

\subsection{Proof of Proposition \ref{pp} for $Q^\ell$}
In this subsection, we demonstrate how to extend the argument for $Q^1$ in the previous subsection to $Q^\ell$, with $\ell<1$. Recall that $Q^\ell$ was defined before as a rectangle of dimension $\ell\times 1$ with $\ell<1$. Let us also recall Proposition \ref{pp}.
\begin{prop}
    For every $\epsilon>0$, and any $f_1$ and $f_2$ satisfying the support separation condition, we have
    \begin{equation}\label{ql}
    \begin{aligned}
\||Ef_1Ef_2|^{1/2}\|_{L^{13/4}(B_R)}&\le C_\epsilon R^{10\epsilon}\big(\|f_1\|_2^{1/2}\|f_2\|_2^{1/2}\big)^{12/13+\epsilon}\\&\big(\max_{\theta:R^{-1/2}-caps}\|f_1\|^{1/2}_{L_{avg}^2(\theta)}\max_{\theta:R^{-1/2}-caps}\|f_2\|^{1/2}_{L_{avg}^2(\theta)}\big)^{1/13-\epsilon}.
    \end{aligned}
    \end{equation}
\end{prop}
We follow a similar approach as [\cite{SS21}, Lemma 4.1]. The idea is to divide the argument into two cases, $R\lesssim \ell^{-2}$ and $R\gg \ell^{-2}$. First consider the case with $R\lesssim \ell^{-2}$. Since we have the localization on the physical side, it allows us to blur on the frequency side. Therefore, we can reduce the extension operator associated with the original surface to another hyperbolic surface over the unit cube.

Precisely, we have the following Lemma. Let $\Sigma_g:=\{(\xi,g(\xi)):\xi\in Q^\ell\}$ be the graph of the hypersurface $g$.

\begin{lemma}\label{error}
    Given a hyperbolic hypersurface $g$ over $Q^\ell$, there exists another hypersurface $\tilde{g}$ that is hyperbolic over the unit cube $Q^1$ and $\Sigma_g$ lies within an $O(\ell^{2})$ neighborhood of $\Sigma_{\tilde{g}}$.
\end{lemma}
\begin{proof}
    With $g=\xi_1^2-\xi_2^2+h(\xi)$, we consider the new surface defined as 
    \[\tilde{g}=\xi_1^2-\xi_2^2+h(0,\xi_2)+\xi_1\partial_1h(0,\xi_2).\]
    Now, we check the condition \eqref{D2} for \[\tilde{h}:=h(0,\xi_2)+\xi_1\partial_1h(0,\xi_2).\] Firstly, $\partial_{11}\tilde{h}$ is the zero function, so the derivative bounds follow directly. 
    
    Secondly, \[\partial_{12}\tilde{h}(\xi_1,\xi_2)=\partial_{12}h(0,\xi_2)\]is a function independent of $\xi_1$ and the boundedness of $\xi_2$ derivatives follows from the hyperbolicity of $g$ and in turn the boundedness of $\xi_2$ derivatives of $h$.

    Lastly, \[\partial_{22}\tilde{h}(\xi_1,\xi_2)=\partial_{22}h(0,\xi_2)+\xi_1\partial_{122}h(0,\xi_2).\] We deal with these two terms separately. The first term is again independent of $\xi_1$ and desired boundedness of $\xi_2$ derivatives follows from the corresponding property of $h$. The second term is linear in $\xi_1$ with the $
    \xi_1$ derivative being $\partial_{122}h(0,\xi_2)=\partial_2(\partial_{12}h(0,\xi_2))$. The boundedness follows due to $g$ being hyperbolic and \eqref{D2}. Similarly, the $\xi_2$ derivatives can be bounded with the same technique of applying $\partial_2$ to $\partial_{12}$ repeatedly. 

    Next, we turn to error estimate, note that $\tilde{g}$ can be viewed as the first order approximation of $g$ in the $\xi_1$ direction and the error can be bounded as
    \[|g(\xi)-\tilde{g}(\xi)|=|h(\xi_1,\xi_2)-h(0,\xi_2)-\xi_1\partial_1h(0,\xi_2)|\lesssim \xi_1^2|\partial_{11}h(\eta,\xi_2)|\lesssim \ell^2,\] for some $\eta$ between 0 and $\xi_1$ and all $\xi\in Q^\ell$. The last inequality follows since $\xi_1\le \ell$ and 
    \[|\partial_{11}h(\eta,\xi_2)|\le \eta\max_{\xi\in Q^\ell}|\partial_{111}h(\xi)|+\max_{\xi\in Q^\ell}|\partial_{112}h(\xi)|\le \ell\max_{\xi\in Q^\ell}|\partial_{111}h(\xi)|+\max_{\xi\in Q^\ell}|\partial_{112}h(\xi)|\lesssim 1.\]
\end{proof}

Now, we return to the case $R\lesssim \ell^{-2}$. Invoking Lemma \ref{error}, the graph of $g$ lies within $O(R^{-1})$ distance from graph of $\tilde{g}$.
We can reduce equation \eqref{ql} for $g$ over $Q^\ell$ to $\tilde{g}$ on the unit cube. Let $\phi\in\mathcal{S}(\R)$ with $|\phi|\gtrsim 1$ on $[-1,1]$ and Fourier supported in $[-1,1]$. Denote $\psi_R(t,x):=\phi(\frac{t}{R})\Pi_{j=1}^2\phi(\frac{x_j}{R})$. Then
\begin{equation}
\begin{aligned}
   &\|E_g^\ell f_1 E_g^\ell f_2\|_{L^{13/4}(B_R)}\lesssim \|(\psi_RE_g^\ell f_1)(\psi_RE_g^\ell f_2)\psi_R\|_{L^{13/4}}  \\&\lesssim \int\int\|E^{\tilde{l}}_{\tilde{g}} f_1^\tau E^{\tilde{l}}_{\tilde{g}} f_2^\tau \psi_R\|_{L^{13/4}}d\tau d\tau'\le C_{1,\epsilon} R^\epsilon \int\int\|f_1^\tau\|_{L^{13/6}}\|f_2^\tau\|_{L^{13/6}}d\tau d\tau'\\& \lesssim C_{1,\epsilon} R^\epsilon \|f_1^\tau\|_{L^{13/6}}\|f_2^\tau\|_{L^{13/6}}.
\end{aligned}
\end{equation}
Here 
$$f_j^\tau(\xi):=\int f_j(\eta)\hat{\psi}_R(\tau+\tilde{g}(\xi)-g(\eta),\xi-\eta)d\eta.$$
Thus, with Young's inequality, we have $\|f_j^\tau\|_{L^{13/6}}\lesssim R\chi_{[-\frac{C}{R},\frac{C}{R}]}\|f_j\|_{L^{13/6}}$. The second to last inequality uses the bilinear estimate for the unit cube, so we can get the base case estimate uniformly in terms of $\ell$ with the implicit constant depending only on $C_{1,\epsilon}$.

It remains to prove \eqref{ql} for $R\gg \ell^{-2}$. We take $R=C\ell^{-2}$ as the base case and the induction on scale argument for $Q^1$ still works. In particular, we can check that the corresponding wave packet decomposition obeys the expected decay by the same proof of [\cite{G18}, Lemma 3.1]. Here, we need $R\gg \ell^{-2}$ to compensate for the derivatives of error term $h$ in \eqref{D2}. Also, we still have the transversality condition between the wave packets from the two separated patches since the second derivatives of the error term $h$ are small \eqref{D2}. Thus, by taking the larger implicit constant in $R\lesssim \ell^{-2}$ and $R\gg \ell^{-2}$ range, we proved Proposition \ref{pp}.

As a corollary, the proof for Theorem \ref{bimain} is complete. We apply this estimate in next section to get Theorem \ref{rectmain} and \ref{cor}.

\section{Restriction over rectangles}
In this section, we turn to the proof of Theorem \ref{rectmain}. First, notice that Equations \eqref{main1} and \eqref{main2} are invariant under parabolic rescaling. Thus, without loss of generality, we may assume that one of the side-lengths of the rectangle is 1. Indeed, consider the hyperbolic surface $g$ over $Q^{\ell_1,\ell_2}$ and suppose we have proved Theorem \ref{rectmain} for rectangle $Q^{r\ell_1,r\ell_2}$, with $r>0$. We can calculate
\begin{equation}
    \begin{aligned}
        E^{\ell_1,\ell_2}_g f(x,t)&=\int_{Q^{\ell_1,\ell_2}} f(\xi)e^{2\pi i(x\cdot \xi+t g(\xi_1,\xi_2))}d\xi\\&=r^{-2}\int_{Q^{r\ell_1,r\ell_2}} f(r^{-1}\eta)e^{2\pi i(r^{-1}x\cdot \eta+r^{-2}t r^2g(r^{-1}\eta_1,r^{-1}\eta_2))}d\eta\\&=r^{-2}E^{r\ell_1,r\ell_2}_{\tilde{g}}\tilde{f}(r^{-1}x,r^{-2}t).
    \end{aligned}
\end{equation}
Here $\tilde{f}(\eta):=f(r^{-1}\eta)$ and $\tilde{g}(\eta):=r^2 g(r^{-1}\eta)$ is hyperbolic over $Q^{r\ell_1,r\ell_2}$. Thus,
\begin{equation}\label{rescale}
    \sup_{f}\frac{\|E^{\ell_1,\ell_2}_g f\|_{L^q}}{\|f\|_{L^p}}=r^{\frac{2}{p}-2}r^{\frac{4}{q}}\sup_{\tilde{f}}\frac{\|E_{\tilde{g}}^{r\ell_1,r\ell_2}\|_{L^q}}{\|\tilde{f}\|_{L^p}}.
\end{equation}
 Invoking Theorem \ref{rectmain} for $Q^{r\ell_1,r\ell_2}$, and depending on where $(p,q)$ live, we apply either \eqref{main1} or \eqref{main2} to bound $\sup_{\tilde{f}}\frac{\|E_{\tilde{g}}^{r\ell_1,r\ell_2}\|_{L^q}}{\|\tilde{f}\|_{L^p}}$. After rearranging the powers, $r$ gets canceled out in both scenarios.

We will consider the following 2 regions. Let $T_1$ denote the closed triangle in the Riesz diagram formed by $1/q=0$, $q=3p'$, and $p=q$, excluding the $p=q$ line; and $T_2$ denote the open quadrilateral in the Riesz diagram formed by $q=3p'$, $q=2p'$, $p=q$, and $q=13/4$.

Following the same strategy as \cite{SS21}, we will prove Theorem \ref{rectmain} in the following 2 subsections, corresponding to the 2 regions.

\subsection{Proof for $(p,q)\in T_1$}
In the region $T_1$, the bound will take the form of $(\ell_1\ell_2^\theta)^{\frac{1}{p'}-\frac{1}{q}}$, with $\theta$ determined by the ``slope" $m$ such that $q=mp'$.

First note that for the line $q=\infty$, inequality \eqref{main2} in Theorem \ref{rectmain} reduces to
$$|\E^\ell_g(x,t)|\le(\ell_1\ell_{2})^{\frac{1}{p'}}\|g\|_p, \,\text{for all}\, (x,t),$$
which follows from the H\"older's inequality.

Using the interpolation result the proof of [\cite{SS21}, Lemma 4.3], it suffices to prove the bound for $(p,q)$ lying on the line $q=3p'$, where the desired bound has the form $\ell_1^{\frac{1}{p'}-\frac{1}{q}}$.

We use the slicing argument in [Lemma 4.2, \cite{SS21}] to get
\begin{equation}
    \begin{aligned}
        \|\E^\ell_gf\|_q&\le\|\E^{\ell_2}_{g_{\xi_1}}f_{\xi_1}\|_{L_{t,x_2}^q(L^{q'}_{\xi_1})}\le\|\E^{\ell_2}_{g_{\xi_1}}f_{\xi_1}\|_{L_{\xi_1}^{q'}(L^{q}_{t,x_2})}\lesssim\|\E^{\ell_2}_{g_{\xi_1}}\|_{L^p\rightarrow L^q}\|f\|_{L_{\xi_1}^{q'}(L^{p}_{\xi_2})}\\&\le\|\E^{\ell_2}_{g_{\xi_1}}\|_{L^p\rightarrow L^q}|\ell_1|^{\frac{1}{q'}-\frac{1}{p}}\|f\|_p.
    \end{aligned}
\end{equation}
The first inequality is the Hausdorff-Young inequality; the second one is Minkowski's inequality, and the last one is H\"older's inequality, with $g_{\xi_1}(\xi_2):=g(\xi_1,\xi_2)=\xi_1^2-\xi_2^2+h(\xi_1,\xi_2)$. Thus, it suffices to show that for all $\xi_1$, we have \begin{equation}\label{para}\|\E^{\ell_2}_{g_{\xi_1}}\|_{L^p\rightarrow L^q}\lesssim 1.\end{equation}

Note that the graph $\big(\xi_2,g_{\xi_1}(\xi_2)\big)$ is a perturbation of the parabola and $(p,q)$ lives on the scaling line. After parabolic rescaling in the second coordinate, inequality \eqref{para} can be proved using the restriction theorem for perturbed parabola \cite{T75,Z74}.

\subsection{Proof for $(p,q)\in T_2$}
In the region $T_2$, the bound will take the form of $(\ell_1^\theta)^{\frac{1}{p'}-\frac{1}{q}}$, with $\theta$ determined by the ``slope" $m$ such that $q=mp'$. Also, since the bound is invariant under parabolic rescaling, we can always rescale the rectangle so that the shorter side $\ell_1=1$. Thus, it suffices to prove that for rectangles $Q^\ell$ of dimension $1\times \ell_2$, and $g$ being hyperbolic over $Q^\ell$, we have $\|\E^{\ell}_{g}\|_{L^p\rightarrow L^q}\lesssim 1.$ Using real interpolation with the result proved in the previous subsection ($\|\E^{\ell}_{g}\|_{L^p\rightarrow L^q}\lesssim 1$ for $(p,q)$ on the line $q=3p'$), it suffices to prove the same bound for all the $(p,q)$ satisfying $13/4< q <10/3$, $q>p$ and $q>2p'$.

Using the Marcinkiewicz interpolation theorem, it suffices to prove a restricted weak type estimate for $f=\chi_\Omega$, with $\Omega\subset Q^\ell$.  

We can apply a Whitney decomposition to get 
$$Q^\ell\times Q^\ell=\bigcup_{N=0}^\infty\bigcup_{\tau\sim\tau'\in\D_N}\tau\times\tau',$$ where $\D_N$ denotes the finitely overlapping family of rectangles of dimension $1\times 2^N$ and $\tau\sim\tau'$ if $N=0$ and dist$(\tau,\tau')\lesssim1$ or $N\ge1$ and dist$(\tau,\tau')\sim 2^N.$ 
From this, we get
\begin{equation}(\E f_\Omega)^2=\sum_{N=0}^\infty\sum_{\tau\sim\tau'\in\D_N}\E f_{\Omega\cap\tau}\E f_{\Omega\cap\tau'}.
\end{equation}
Similar to the argument in \cite{SS21}, we get
\begin{equation}\label{m}
\|\E f_\Omega\|_q^q\lesssim\sum_{N=0}^\infty\sum_{\tau\sim\tau'\in\D_N}\|\E f_{\Omega\cap\tau}\E f_{\Omega\cap\tau'}\|_{\frac{q}{2}}^{\frac{q}{2}}.    
\end{equation}
Using H\"older's inequality, and the boundedness of $\E^\ell_g$ over rectangles comparable to the unit ball from \cite{BMV23}, Theorem 1.1, we get
\begin{equation}
\sum_{\tau\sim\tau'\in\D_0}\|\E f_{\Omega\cap\tau}\E f_{\Omega\cap\tau'}\|_{\frac{q}{2}}^{\frac{q}{2}}\lesssim\sum_{\tau\in\D_0}|\Omega\cap\tau|^{\frac{q}{p}}\lesssim |\Omega|^{\frac{q}{p}}.    
\end{equation}
Thus, we only need to consider those terms with $N>0$. We will use the rescaled bilinear estimate at $(p_q,q)$. $p_q$ will be chosen later so that this exponent pair lies sufficiently close to the bilinear scaling line and satisfies
\begin{equation}\label{com}\frac{q}{p'_q}\le \frac{q}{p}+4-q.\end{equation}Note that this is achievable since we have bilinear estimates for all the $(p,q)$ with $q>13/4$ and $5/q+1/p<2$. We can calculate that the last inequality is the same as 
$$\frac{q}{p'}+q>5.$$
Comparing this with \eqref{com}, and using the fact that $q>p$, we can choose $p_q'$ so that
$$5<\frac{q}{p'_q}+q\le \frac{q}{p}+4.$$
Thus, for those terms with $\tau\sim\tau'\in\D_N$, using the rescaled Theorem \ref{bimain}, we get
\begin{equation}
\|\E f_{\Omega\cap\tau}\E f_{\Omega\cap\tau'}\|_{\frac{q}{2}}\lesssim 2^{N(\frac{4}{p'_q}-\frac{8}{q})}\|f_{\Omega\cap\tau}\|_{p_q}\|f_{\Omega\cap\tau'}\|_{p_q}.    
\end{equation}
Now, we continue the calculation in Equation \eqref{m}
\begin{equation}\label{long}
    \begin{aligned}
&\sum_{N=1}^\infty\sum_{\tau\sim\tau'\in\D_N}\|\E f_{\Omega\cap\tau}\E f_{\Omega\cap\tau'}\|_{\frac{q}{2}}^{\frac{q}{2}}\lesssim  \sum_{N=1}^\infty 2^{N\big(\frac{2q}{p'_q}-4\big)}\sum_{\tau\sim\tau'\in\D_N}|\Omega\cap\tau|^{\frac{q}{2p_q}}|\Omega\cap\tau'|^{\frac{q}{2p_q}}\\&\lesssim \sum_{N=1}^\infty 2^{N\big(\frac{2q}{p'_q}-4\big)}\sum_{\tau\in\D_N}|\Omega\cap\tau|^{\frac{q}{p_q}}\lesssim \sum_{N=1}^\infty 2^{N\big(\frac{2q}{p'_q}-4\big)}\min(|\Omega|,2^{N})^{\frac{q}{p_q}-1}|\Omega|.
    \end{aligned}
\end{equation}
There are two cases, divided by whether $|\Omega|<1$ or not.

If $|\Omega|<1$, first note that the power in the front of the last term of equation \eqref{long} is negative.
Indeed, this is equivalent to \begin{equation}\label{1}
    q<2 p'_q.
\end{equation} This is true since $(p_q,q)$ is chosen sufficiently close to the bilinear restriction line, which is above the linear scaling line $q=2p'$. 
Thus, we can calculate that
\begin{equation}
\sum_{N=1}^\infty\sum_{\tau\sim\tau'\in\D_N}\|\E f_{\Omega\cap\tau}\E f_{\Omega\cap\tau'}\|_{\frac{q}{2}}^{\frac{q}{2}}\lesssim |\Omega|^{\frac{q}{p_q}}< |\Omega|^{\frac{q}{p}}, 
\end{equation}
since $|\Omega|<1$ and $p>p_q$.

Next, we deal with the case that $|\Omega|\ge1.$ We divide the last term of Equation \eqref{long} according to whether $|\Omega|>2^N$. Thus, 
\begin{equation}\label{last}
    \begin{aligned}
&\sum_{N=1}^\infty\sum_{\tau\sim\tau'\in\D_N}\|\E f_{\Omega\cap\tau}\E f_{\Omega\cap\tau'}\|_{\frac{q}{2}}^{\frac{q}{2}} \\&\lesssim \sum_{N=1}^{\log|\Omega|} 2^{N\big(\frac{2q}{p'_q}-5+\frac{q}{p_q}\big)}|\Omega|+\sum_{N=\log|\Omega|}^\infty 2^{N\big(\frac{2q}{p'_q}-4\big)} |\Omega|^{\frac{q}{p_q}}
    \end{aligned}
\end{equation}
As calculated in equation \eqref{1}, we know that the power in the second term is negative. Thus, the second term can be bounded by \begin{equation}\label{2}
2^{\log|\Omega|\big(\frac{2q}{p'_q}-4\big)} |\Omega|^{\frac{q}{p_q}}=|\Omega|^{\big(q+\frac{q}{p'_q}-4\big)}.
\end{equation}
Then, we turn to the first term in Equation \eqref{last}, which is a geometric series. Thus, it is bounded by the last term in the sum, which gives $|\Omega|^{\big(q+\frac{q}{p'_q}-4\big)}$, the same as the second term. 
Therefore, it suffices to show that 
$$|\Omega|^{\big(q+\frac{q}{p'_q}-4\big)}\le|\Omega|^{q/p}.$$
Since $|\Omega|\ge 1$, it suffices to show 
$$q+\frac{q}{p'_q}-4\le\frac{q}{p},$$
which follows exactly from our choice of $p'_q$.

Now, the proof for Theorem \ref{rectmain} is complete.

\subsection{Proof of Theorem \ref{cor}}

In this subsection, we deduce Theorem \ref{cor} from Theorem \ref{rectmain}. Given rectangle $R^{\ell,\phi}$ and function $f$ supported in $R^{\ell,\Phi}$, recall that we defined the linear mappings $L: \R^2\rightarrow \R^2: (\xi_1,\xi_2)\mapsto (\xi_1,(\tan\phi)^{-1}\xi_2)$, and $M: \R^2\rightarrow \R^2: (\xi_1,\xi_2)\mapsto \big(1/2(\xi_1+\xi_2),1/2(-\xi_1+\xi_2)\big)$. 

We can calculate that the image of $R^{\ell,\phi}$ under the mapping defined before is $M\circ L(R^{\ell,\phi})=R^{\tilde{\ell},0}$, with $\tilde{\ell}\sim (\min(\ell_1,\cos\phi\cdot\ell_2),\max(\ell_1,\cos\phi\cdot\ell_2))$, i.e. $M\circ L(R^{\ell,\phi})$ is an axis parallel rectangle of dimension $\tilde{\ell}$. Below is a sketch of the transformations in the two-step reduction to the axis-parallel rectangles. 

\begin{figure}[h]
    \centering
    \includegraphics[width=0.8\textwidth]{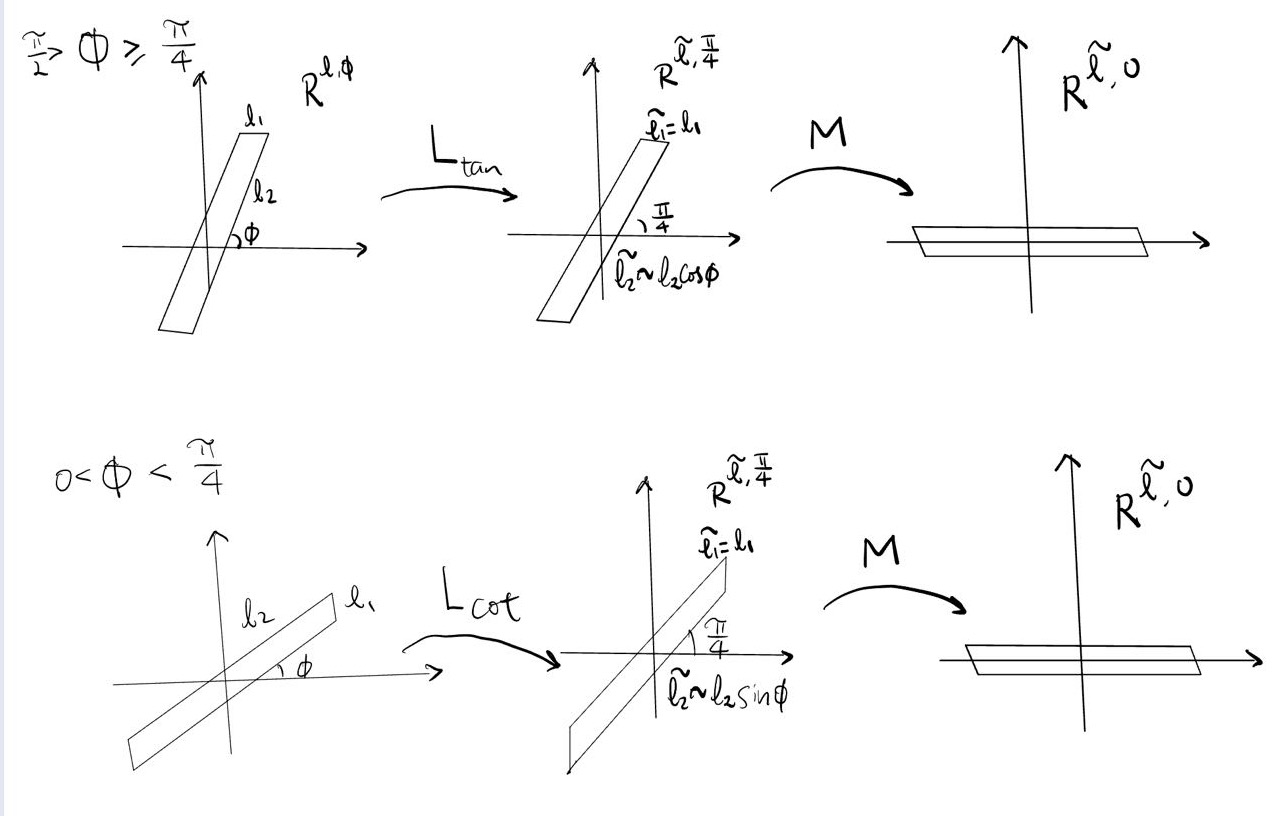}
    \caption{Sketch of the transformations}
\end{figure}
 Next, we keep track track of how the operator norms change under these transformations. First consider the linear mapping $L$ that maps $R^{\ell,\Phi}$ to $R^{\tilde{\ell},\frac{\pi}{4}}$, $\tilde{\ell}\sim (\min(\ell_1,\cos\phi\cdot\ell_2),\max(\ell_1,\cos\phi\cdot\ell_2))$. Thus, we can compute that 
\begin{equation}
\begin{aligned}
    E^g_\ell f(x,t)&=\int_{R^{\ell,\Phi}}f(\xi)e^{2\pi i(x\cdot\xi+t(\xi_1\xi_2)+th(\xi))}d\xi\\&=(\tan\phi)\int_{R^{\tilde{\ell},\frac{\pi}{4}}}f(L^{-1}\eta)e^{2\pi i(x\cdot L^{-1}\eta+(\tan\phi)t(\eta_1\eta_2)+(\tan\phi)t(\tan\phi)^{-1}h(L^{-1}\eta))}d\eta\\&=(\tan\phi)E^{\tilde{g}}_{\tilde{\ell}}\tilde{f}(x_1,(\tan\phi)x_2,(\tan\phi)t).
\end{aligned}
\end{equation}
Here $\tilde{f}(\eta):=f(L^{-1}\eta)$, and $\tilde{g}(\eta):=\eta_1\eta_2+(\tan\phi)^{-1}h(\eta_1,(\tan\phi)\eta_2):=\eta_1\eta_2+\tilde{h}(\eta)$.

We can calculate that 
\begin{equation}
    \begin{aligned}
        \|E^g_\ell\|_{L^p\rightarrow L^q}=\sup_f\frac{\|E^g_\ell f\|_{L^q}}{\|f\|_{L^p}}=(\tan\phi)^{\frac{1}{p'}-\frac{2}{q}}\sup_{\tilde{f}}\frac{\|E^{\tilde{g}}_{\tilde{\ell}} \tilde{f}\|_{L^q}}{\|\tilde{f}\|_{L^p}}=(\tan\phi)^{\frac{1}{p'}-\frac{2}{q}}\|E^{\tilde{g}}_{\tilde{\ell}}\|_{L^p\rightarrow L^q}.
    \end{aligned}
\end{equation}

Now, it suffices to deal with $\|E^{\tilde{g}}_{\tilde{\ell}}\|_{L^p\rightarrow L^q}$. We apply the rotation $M$ to calculate
\begin{equation}
    \begin{aligned}
  E^{\tilde{g}}_{\tilde{\ell}} f(x,t)&=\int_{R^{\tilde{\ell},\frac{\pi}{4}}}f(\xi)e^{2\pi i(x\cdot\xi+t(\xi_1\xi_2)+t\tilde{h}(\xi))}d\xi\\&\sim\int_{R^{\tilde{\ell},0}}f(M^{-1}\eta)e^{2\pi i(x\cdot M^{-1}\eta+t(\eta_1^2-\eta_2^2)+t\tilde{h}(M^{-1}\eta))}d\eta\\&=E^{\tilde{\tilde{g}}}_{\tilde{\ell}}\tilde{f}(M^{-1}x,t),      
    \end{aligned}
\end{equation}
with $\tilde{\tilde{g}}(\eta):=\eta_1^2-\eta_2^2+\tilde{h}(M^{-1}\eta).$ Note that this function is hyperbolic over the axis parallel rectangle $R^{\tilde{\ell},0}$. Also $\tilde{f}:=f\circ M^{-1}$ and $f$ have comparable $L^p$ norms. Thus, we can compute
\begin{equation}
    \begin{aligned}
\|E^{\tilde{g}}_{\tilde{\ell}}\|_{L^p\rightarrow L^q}=\sup_f\frac{\|E^{\tilde{g}}_{\tilde{\ell}} f\|_{L^q}}{\|f\|_{L^p}}\sim\sup_{\tilde{f}}\frac{\|E^{\tilde{\tilde{g}}}_{\tilde{\ell}} \tilde{f}\|_{L^q}}{\|\tilde{f}\|_{L^p}}=\|E^{\tilde{\tilde{g}}}_{\tilde{\ell}}\|_{L^p\rightarrow L^q}.
    \end{aligned}
\end{equation}

Now, applying Proposition \ref{rectmain} to $\|E^{\tilde{\tilde{g}}}_{\tilde{\ell}}\|_{L^p\rightarrow L^q}$, the proof for Theorem \ref{cor} is complete.

\section{Proof of Proposition \ref{omit1} and \ref{omit2}}
In this section, we deduce Proposition \ref{omit1} and \ref{omit2} from Theorem \ref{rectmain} and \ref{cor}. Since the passage from Theorem \ref{cor} to Proposition \ref{omit2} is the same as that from Theorem \ref{rectmain} to Proposition \ref{omit1}, we only present the latter. 

First observe that \eqref{q<p1} and \eqref{q<p2} are both invariant under parabolic rescaling. Indeed, we reuse the argument in the beginning of Section 3, in particular \eqref{rescale}. After rearranging the powers, again the $r$ gets canceled out.

Thus, without loss of generality, we may assume the dimension of $Q^\ell$ is $1\times \ell$ with $\ell\ge1$. We first prove \eqref{q<p1}:
\begin{equation}
        \|\E^\ell_g\|_{L^p\rightarrow L^q}\lesssim_\epsilon (\ell)^{\frac{1}{q}-\frac{1}{p}}\ell^\epsilon,
    \end{equation}
    for $p\ge q,\frac{13}{4}<q\le 4$. This can be deduced from \eqref{main1} using H\"older's inequality twice.

Indeed, given $\epsilon>0$, we will choose $p_\epsilon<q$ sufficiently close to $q$. Under the assumption $\ell_1=1$, \eqref{main1} becomes $\|\E^\ell_g\|_{L^{p_\epsilon}\rightarrow L^q}\lesssim_\epsilon 1$. Thus, for all $f\in L^p(Q^\ell)$,
\begin{equation}
    \begin{aligned}
        \|E^\ell_gf\|_{L^q}\lesssim_\epsilon \|f\|_{L^{p_\epsilon}}\le \|f\|_{L^q}|Q^\ell|^{\frac1{p_\epsilon}-\frac1 q}\le \|f\|_{L^p}|Q^\ell|^{\frac1{p_\epsilon}-\frac1 q}|Q^\ell|^{\frac1{q}-\frac1p}.
    \end{aligned}
\end{equation}
It suffices to note that $|Q^\ell|=\ell$ and pick $p_\epsilon$ so that $\frac1{p_\epsilon}-\frac1 q<\epsilon$.

Next, we turn to \eqref{q<p2}:
\begin{equation}
        \|\E^\ell_g\|_{L^p\rightarrow L^q}\lesssim \ell^{1-\frac{3}{q}-\frac{1}{p}}.
    \end{equation}

    Invoking \eqref{main2} on the line $q=3p'$ and H\"older's inequality, we get, for all $f\in L^p(Q^\ell)$,
    \[\|E^\ell_gf\|_{L^q}\lesssim \|f\|_{L^{(\frac q3)'}}\le \|f\|_{L^p}|Q^\ell|^{1-\frac 3q-\frac1p}=\ell^{1-\frac 3q-\frac1p}\|f\|_{L^p}.\]

\section{Application: Extension estimate for certain degenerate surfaces with additive structure}
In this section, we provide a sketch for the proof for Proposition \ref{deg}. For a complete proof, please refer to \cite{SS21}, Theorem 1.7. Here we highlight some key steps to illustrate the idea of how to use Theorem \ref{rectmain} to get estimates for degenerate surfaces of the form $|\xi_1|^{\beta_1}-|\xi_2|^{\beta_2}$. 

We introduce the notation 
$$(0,1]^2=\bigcup_{k\in\N^2} R^k,\quad R^k:=\{\xi:\xi_i\sim 2^{-k_i},i=1,2\};$$
$$E_\beta=\sum_{k\in\N^2}E^k_\beta,\quad E_\beta^k f(t,x):=\int_{R^k}f(\xi)e^{2\pi i (x\cdot \xi+t(|\xi_1|^{\beta_1}-|\xi_2|^{\beta_2}))}d\xi.$$
We will first upper bound $\|E^k_\beta\|_{L^p\rightarrow L^q}$ and then use the triangle inequality to sum up all the $k\in\N^2$. By symmetry, it suffices to sum up those $k\in\N^2$ with $k_1\beta_1\ge k_2\beta_2$. Using the change of variable $\eta_i:=2^{(1-\frac{\beta_i}{2})k_i}\xi_i,i=1,2$, we can calculate 
\begin{equation}
    \begin{aligned}
&E_\beta^k f(t,x)=2^{-(1-\frac{\beta_1}{2})k_1-(1-\frac{\beta_2}{2})k_2}\tilde{E}\tilde{f}(t,2^{-(1-\frac{\beta_1}{2})k_1}x_1,2^{-(1-\frac{\beta_2}{2})k_2}x_2),\\&  \tilde{f}(\eta_1,\eta_2):=f(2^{-(1-\frac{\beta_1}{2})k_1}\eta_1,2^{-(1-\frac{\beta_2}{2})k_2}\eta_2),  
    \end{aligned}
\end{equation}
and $\tilde{E}$ denotes the extension operator associated with the surface 
$$g(\eta):=2^{(\frac{\beta_1}{2}-1)k_1\beta_1}|\eta_1|^{\beta_1}-2^{(\frac{\beta_2}{2}-1)k_2\beta_2}|\eta_2|^{\beta_2},$$ which is hyperbolic over the rectangle $Q_k$ of dimension $\sim 2^{-\frac{k_1\beta_1}{2}}\times 2^{-\frac{k_2\beta_2}{2}}$. 

To illustrate the idea, let us focus on the pair $(p,q)$ in the region where we can apply \eqref{main1} in Theorem \ref{rectmain}. Invoking the bound for $\tilde{E}$ over $Q_k$, we get
\begin{equation}\label{beta}
\begin{aligned}  \|E_\beta^k\|_{L^p\rightarrow L^q}&=\sup_f \frac{\|E^k_\beta f\|_q}{\|f\|_p}=\sup_{\tilde{f}} \frac{(2^{-(1-\frac{\beta_1}{2})k_1-(1-\frac{\beta_2}{2})k_2})^{1-\frac{1}{q}}\|\tilde{E} \tilde{f}\|_q}{(2^{-(1-\frac{\beta_1}{2})k_1-(1-\frac{\beta_2}{2})k_2})^{\frac{1}{p}}\|\tilde{f}\|_p}\\&=(2^{-(1-\frac{\beta_1}{2})k_1-(1-\frac{\beta_2}{2})k_2})^{\frac{1}{p'}-\frac{1}{q}}\|\tilde{E}\|_{L^p\rightarrow L^q}\\&\sim (2^{-\theta\frac{k_1\beta_1}{2}})^{\frac{1}{p'}-\frac{1}{q}}(2^{-(1-\frac{\beta_1}{2})k_1-(1-\frac{\beta_2}{2})k_2})^{\frac{1}{p'}-\frac{1}{q}}
\\&= [2^{k_1\beta_1(\frac{1}{2}-\frac{\theta}{2}-\frac{1}{\beta_1})} 2^{k_2\beta_2(\frac{1}{2}-\frac{1}{\beta_2})}]^{\frac{1}{p'}-\frac{1}{q}}.
\end{aligned}\end{equation}

It remains to sum up in $k$. For simplicity, let us focus on the case with $\beta_1=\beta_2=\beta>2$ Recall that under this case, in Proposition \ref{deg}, the assumption on $p,q$ is $q\ge (1+\frac\beta2)p'$ . 

Note that since $q>2p'$, and $\beta>2$, the power of the second term is positive. We can first sum up $0\le k_2\le k_1$, and upper bound it using the last term $[2^{k_1\beta(\frac{1}{2}-\frac{1}{\beta})}]^{\frac{1}{p'}-\frac{1}{q}}$ and combine this term with the first term. Precisely, by triangle inequality
\begin{equation}\label{BG}
    \begin{aligned}
        \|\sum_{k_2=0}^{k_1}E_\beta^{k_1,k_2}\|_{L^p\rightarrow L^q}&\lesssim[2^{k_1\beta(\frac{1}{2}-\frac{\theta}{2}-\frac{1}{\beta})}]^{\frac{1}{p'}-\frac{1}{q}} [2^{k_1\beta(\frac{1}{2}-\frac{1}{\beta})}]^{\frac{1}{p'}-\frac{1}{q}}\\&=[2^{k_1\beta(\frac{2-\theta}2-\frac{2}{\beta})}]^{\frac{1}{p'}-\frac{1}{q}}\\&=[2^{k_1\beta(\frac{p'}{q-p'}-\frac{2}{\beta})}]^{\frac{1}{p'}-\frac{1}{q}}\\&=2^{-k_1(\frac{2}{p'}-\frac{\beta+2}q)}
    \end{aligned}
\end{equation}
The second to last equality follows from the definition of $\theta$ in the relation $q=(1+\frac{2}{2-\theta})p'$, and the last equality is just arithmetic simplification. Note that the assumption
$q\ge(1+\frac{\beta}{2})p'$ lands exactly on the boundary of summability for $k_1$, so we use the following Bourgain summation trick [\cite{BORSS22}, Lemma 2.8].

\begin{lemma}\label{Bsum}
Let $\{T_j\}_{j\in\Z}$ be a family of linear operators, and assume that
for some exponents
\[1\le p_0,p_1\le\infty,\quad 1\le q_0,q_1\le\infty,\]
and some constants $C_0,C_1>0$, $\gamma_0,\gamma_1>0$, one has
\[\|T_j\|_{L^{p_0}\rightarrow L^{q_0}}\le C_02^{j\gamma_0},\]
and
\[\|T_j\|_{L^{p_1}\rightarrow L^{q_1}}\le C_12^{-j\gamma_1}\]
for every $j\in\Z$.
Set
\[\theta=\frac{\gamma_0}{\gamma_0+\gamma_1},\]
and define $p_\theta,q_\theta$ by
\[\frac1{p_\theta}=\frac{1-\theta}{p_0}+\frac{\theta}{p_1},\quad\frac1{q_\theta}=\frac{1-\theta}{q_0}+\frac{\theta}{q_1}.\]
Then
\[T=\sum_{j\in\Z}T_j\]
is of restricted weak type $(p_\theta,q_\theta)$, i.e.
\[\|T\chi_E\|_{L^{q_\theta,\infty}}\lesssim C_0^{1-\theta}C_1^\theta|E|^{1/p_\theta}.\]
\end{lemma}

We apply the previous lemma with $T_j:=\sum_{k_2=0}^jE_\beta^{j,k_2}$. For the target pair $(p,q)$ with $q=(1+\frac\beta2)p'$ and $q>13/4$, we take 2 pairs $(p_0,q_0)$ and $(p_1,q_1)$ sufficiently close to $(p,q)$ so that we can apply \eqref{BG} to get
\[\|T_{k_1} \|_{L^{p_0}\rightarrow L^{q_0}}\le C_02^{-{k_1}(\frac{2}{p_0'}-\frac{\beta+2}{q_0})},\]
and
\[\|T_{k_1}\|_{L^{p_1}\rightarrow L^{q_1}}\le C_12^{-{k_1}(\frac{2}{p_1'}-\frac{\beta+2}{q_1})}\]
for every $k_1\in\N$. We pick the interpolating pairs with $q_0<(\frac\beta2+1)p_0'$ and $q_1>(\frac\beta2+1)p_1'$ so that $\gamma_0=-(\frac{2}{p_0'}-\frac{\beta+2}{q_0})>0$ and $\gamma_1=\frac{2}{p_1'}-\frac{\beta+2}{q_1}>0$. In addition, we assume that these pairs satisfy the property that in the Riesz diagram, the line segment formed by $(\frac1{p_0},\frac1{q_0})$ and $(\frac1{p_1},\frac1{q_1})$ intersects the line $q=(\frac\beta2+1)p'$ exactly at $(\frac1{p},\frac1{q})$. Note that $(p_\theta,q_\theta)$ satisfy 
\[\frac1{p_\theta}=\frac{1-\theta}{p_0}+\frac\theta{p_1};\quad \frac1{q_\theta}=\frac{1-\theta}{q_0}+\frac\theta{q_1},\] with $\theta$ satisfying $-(1-\theta)\gamma_0+\theta\gamma_1=0.$ This implies
\begin{equation}
    \begin{aligned}
&(1-\theta)\bigg(\frac{2}{p_0'}-\frac{\beta+2}{q_0}\bigg)+\theta \bigg(\frac{2}{p_1'}-\frac{\beta+2}{q_1}\bigg)\\&=(1-\theta)\frac2{p_0'}+\theta\frac2{p_1'}  - (1-\theta)\frac{\beta+2}{q_0}-\theta\frac{\beta+2}{q_1}\\&=\frac2{p'}-\frac{\beta+2}{q}=0. 
    \end{aligned}
\end{equation}
On the line segment formed by $(\frac1{p_0},\frac1{q_0})$ and $(\frac1{p_1},\frac1{q_1})$, the only pair of $(\frac1p,\frac1q)$ satisfying the above relation is the target pair, so that we can apply Bourgain summation Lemma to conclude that $E_\beta\sim \sum_{k_1=0}^\infty \sum_{k_2=0}^{k_1}E_\beta^{k_1,k_2}$ is of weak type $(p,q)$ for all $q=(\frac\beta2+1)p'$ and $q>13/4$. Noting the open endpoint condition $q>13/4$, we can also conclude $E_\beta$ is of strong type $(p,q)$ by the restricted weak type Marcinkiewicz interpolation.

Lastly, we show the necessity of condition $\frac{q}{p'}\ge\max\big(1+\frac{1}{\frac{1}{\max(\beta_1,\beta_2)}+\frac{1}{2}},1+\frac{1}{\frac{1}{\beta_1}+\frac{1}{\beta_2}}\big)$. Testing Knapp example with $f:=\chi_R$ with $R$ being an rectangle of dimension $\delta^{\frac{1}{\beta_1}}\times \delta^{\frac{1}{\beta_2}}$, and sending $\delta$ to 0, we get the necessity of the second argument. 

For the first component, suppose that $\frac{q}{p'}< 1+\frac{1}{\frac{1}{\max(\beta_1,\beta_2)}+\frac{1}{2}}$. Without loss of generality, let us assume that $\beta_1\ge \beta_2$. We are going to reuse the computation \eqref{beta} since $1+\frac{1}{\frac{1}{\max(\beta_1,\beta_2)}+\frac{1}{2}}\le 3$, and we can apply \eqref{main1} to lower bound $\|\tilde{E}\|_{L^p\rightarrow L^q},$ and get \begin{equation}\label{beta2}
\begin{aligned}  \|E_\beta\|_{L^p\rightarrow L^q}\ge \|E_\beta^k\|_{L^p\rightarrow L^q}&=\sup_f \frac{\|E^k_\beta f\|_q}{\|f\|_p}=\sup_{\tilde{f}} \frac{(2^{-(1-\frac{\beta_1}{2})k_1-(1-\frac{\beta_2}{2})k_2})^{1-\frac{1}{q}}\|\tilde{E} \tilde{f}\|_q}{(2^{-(1-\frac{\beta_1}{2})k_1-(1-\frac{\beta_2}{2})k_2})^{\frac{1}{p}}\|\tilde{f}\|_p}\\&=(2^{-(1-\frac{\beta_1}{2})k_1-(1-\frac{\beta_2}{2})k_2})^{\frac{1}{p'}-\frac{1}{q}}\|\tilde{E}\|_{L^p\rightarrow L^q}\\&\gtrsim (2^{-\theta\frac{k_1\beta_1}{2}})^{\frac{1}{p'}-\frac{1}{q}}(2^{-(1-\frac{\beta_1}{2})k_1-(1-\frac{\beta_2}{2})k_2})^{\frac{1}{p'}-\frac{1}{q}}
\\&= [2^{k_1\beta_1(\frac{1}{2}-\frac{\theta}{2}-\frac{1}{\beta_1})} 2^{k_2\beta_2(\frac{1}{2}-\frac{1}{\beta_2})}]^{\frac{1}{p'}-\frac{1}{q}}.
\end{aligned}\end{equation} We pick $k_2=0$ and note that with $$q=(1+\frac{2}{2-\theta})p',\quad 2p'<q<(1+\frac{1}{\frac{1}{2}+\frac{1}{\beta_1}})p',$$ we have $$1+\frac{2}{2-\theta}<1+\frac{1}{\frac{1}{2}+\frac{1}{\beta_1}},$$ which simplifies to $1-\frac{2}{\beta_1}-\theta>0.$ Thus, in the exponent of the the last term of \eqref{beta2}, the coefficient of $k_1$ is positive. Letting $k_1$ go to infinity, we proved that $E_\beta$ is unbounded for $\frac{q}{p'}< 1+\frac{1}{\frac{1}{\max(\beta_1,\beta_2)}+\frac{1}{2}}$.

\end{document}